\documentclass[reqno,openright,dvips,11pt,myheadings,twoside]{amsart}

\usepackage{ifthen}

\usepackage{amssymb}

\usepackage[top=1.4in, bottom=1.4in, left=1.75in, right=1.75in]{geometry} 

\newcounter{arXivFormat} 

\newcommand{\IfarXivElse}[2]{
    \ifthenelse{\value{arXivFormat}=1}
        {#1}{#2}
    }

\newcounter{bIncludeConstantinRamosSection} 

\newcounter{dtlForSubmission} 
\newcounter{dtlMarginComments} 
\newcounter{dtlSomeDetail} 
\newcounter{dtlFullDetails} 

\newcounter{DetailLevel} 

\newcommand{\DetailForSubmission}[1]{
    \ifthenelse{\value{DetailLevel}=\value{dtlForSubmission}}
        {{\small #1}}{}
    }

\newcommand{\DetailMarginNote}[1]{
    \ifthenelse{\value{DetailLevel}=\value{dtlMarginComments} \or \value{DetailLevel}>\value{dtlMarginComments}}
        {{\small #1}}{}
    }

\newcommand{\DetailSome}[1]{
    \ifthenelse{\value{DetailLevel}=\value{dtlSomeDetail} \or \value{DetailLevel}>\value{dtlSomeDetail}}
        {{\small \textbf{Detailed compile only}: #1}}{}
    }

\newcommand{\DetailFull}[1]{
    \ifthenelse{\value{DetailLevel}=\value{dtlFullDetails} \or \value{DetailLevel}>\value{dtlFullDetails}}
        {{\small \textbf{Detailed compile only}: #1}}{}
    }

\newcommand{\NotDetailSome}[1]{
    \ifthenelse{\value{DetailLevel}=\value{dtlSomeDetail} \or \value{DetailLevel}>\value{dtlSomeDetail}}
        {}{#1}
    }

\newcommand{\NotDetailFull}[1]{
    \ifthenelse{\value{DetailLevel}=\value{dtlFullDetails} \or \value{DetailLevel}>\value{dtlFullDetails}}
        {}{#1}
    }

\newcommand{\DetailSomeElse}[2]{
    \ifthenelse{\value{DetailLevel}=\value{dtlSomeDetail} \or \value{DetailLevel}>\value{dtlSomeDetail}}
        {{\small \textbf{Detailed compile only}: #1}}{#2}
    }

\newcommand{\DetailFullElse}[2]{
    \ifthenelse{\value{DetailLevel}=\value{dtlFullDetails} \or \value{DetailLevel}>\value{dtlFullDetails}}
        {{\small \textbf{Detailed compile only}: #1}}{#2}
    }

\newcommand{\DetailSomeInline}[1]{
    \ifthenelse{\value{DetailLevel}=\value{dtlSomeDetail} \or \value{DetailLevel}>\value{dtlSomeDetail}}
        {{\small #1}}{}
    }

\newcommand{\DetailFullInline}[1]{
    \ifthenelse{\value{DetailLevel}=\value{dtlFullDetails} \or \value{DetailLevel}>\value{dtlFullDetails}}
        {{\small #1}}{}
    }

\newcommand{\DetailSomeElseInline}[2]{
    \ifthenelse{\value{DetailLevel}=\value{dtlSomeDetail} \or \value{DetailLevel}>\value{dtlSomeDetail}}
        {{\small #1}}{#2}
    }

\newcommand{\DetailFullElseInline}[2]{
    \ifthenelse{\value{DetailLevel}=\value{dtlFullD	etails} \or \value{DetailLevel}>\value{dtlFullDetails}}
        {{\small #1}}{#2}
    }

\newcommand{\ExplainDetailLevel}{
    Detail level is
    \ifthenelse{\value{DetailLevel}=\value{dtlForSubmission}}
        {0: for submission}
        {\ifthenelse{\value{DetailLevel}=\value{dtlMarginComments}}
            {1: as for submission but with margin comments}
            {\ifthenelse{\value{DetailLevel}=\value{dtlSomeDetail}}
                {2: some proofs not intended for submission}
               {\ifthenelse{\value{DetailLevel}=\value{dtlFullDetails}}
                   {3: full details}
                   {invalid}
                }
            }
        }
    }

\newtheorem{theorem}{Theorem}[section]
\newtheorem{lemma}[theorem]{Lemma}

\newtheorem{cor}[theorem]{Corollary}

\theoremstyle{definition}
\newtheorem{definition}[theorem]{Definition}

\newtheorem{remark}[theorem]{Remark}

\numberwithin{equation}{section}

\newcommand{\abs}[1]{\left\vert#1\right\vert}

\newcommand{\innp}[1]{\ensuremath{\left< #1 \right>}}

\newcommand{\skipline}{\vspace{11pt}}

\newcommand{\BB}[1]{\ensuremath{\mathbb{#1}}}
\newcommand{\R}{\ensuremath{\BB{R}}} %
\newcommand{\Test}{\ensuremath{\Cal{T}}} %
\newcommand{\iny}{\ensuremath{\infty}}
\newcommand{\grad}{\ensuremath{\nabla}}
\DeclareMathOperator{\dv}{div} %
\DeclareMathOperator{\supp}{supp} %
\DeclareMathOperator{\Lim}{Lim} %

\newcommand{\prt}{\ensuremath{\partial}}
\newcommand{\brac}[1]{\ensuremath{\left[ #1 \right]}}
\newcommand{\pr}[1]{\ensuremath{\left( #1 \right) }}
\newcommand{\set}[1]{\ensuremath{\left\{ #1 \right\}}}

\newcommand{\norm}[1]{\ensuremath{\left\Vert #1 \right\Vert}}
\newcommand{\smallnorm}[1]{\ensuremath{\Vert #1 \Vert}}
\newcommand{\refS}[1]{Section~\ref{S:#1}}
\newcommand{\refT}[1]{Theorem~\ref{T:#1}}
\newcommand{\refL}[1]{Lemma~\ref{L:#1}}

\newcommand{\refD}[1]{Definition~\ref{D:#1}}
\newcommand{\refC}[1]{Corollary~\ref{C:#1}}
\newcommand{\refE}[1]{Equation~(\ref{e:#1})}
\newcommand{\refEAnd}[2]{Equations~(\ref{e:#1}) and (\ref{e:#2})}
\newcommand{\refEThrough}[2]{Equations~(\ref{e:#1}) through (\ref{e:#2})}
\newcommand{\refR}[1]{Remark~\ref{R:#1}}
\newcommand{\refRAnd}[2]{Remarks~\ref{R:#1} and \ref{R:#2}}

\newcommand{\Cal}[1]{\ensuremath{\mathcal{#1}}}

\newcommand{\la}{\ensuremath{\lambda}}

\newcommand{\diff}[2]{\frac{ d#1}{d#2}}

\newcommand{\ol}{\overline}
\newcommand{\smallabs}[1]{\ensuremath{\vert #1 \vert}}

\begin{document}

\raggedbottom
% \reversemarginpar

\numberwithin{equation}{section}

%-------------------- "Style" Commands --------------------
%
%
\newcommand{\MarginNote}[1]{
    \ifthenelse{\value{DetailLevel}=\value{dtlMarginComments} \or
            \value{DetailLevel}>\value{dtlMarginComments}} {
        \marginpar{
            \begin{flushleft}
                \footnotesize #1
            \end{flushleft}
            }
        }
        {}
    }

%
% Define a note-to-self, which remains in the TeX document, but
% produces no output. It will disrupt spacing, unfortunately,
% in the same way that MarginNote will.
%
\newcommand{\NoteToSelf}[1]{
    }

%
% Obsolete material, which should eventually simply be deleted.
%
\newcommand{\Obsolete}[1]{
    }

%
% Tentative material, which might eventually lead somewhere..
%
\newcommand{\Tentative}[1]{
    }

\newcommand{\Detail}[1]{
    \MarginNote{Detail}
    \skipline
    \hspace{+0.25in}\fbox{\parbox{4.25in}{\small #1}}
    \skipline
    }

\newcommand{\Todo}[1]{
    \skipline \noindent \textbf{TODO}:
    #1
    \skipline
    }

\newcommand{\Comment}[1] {
    \skipline
    \hspace{+0.25in}\fbox{\parbox{4.25in}{\small \textbf{Comment}: #1}}
    \skipline
    }

\newcommand{\Ignore}[1]{}

%
% Commonly used expressions in this document.
%

%--- Integral over time and space
\newcommand{\IntTR}
    {\int_{t_0}^{t_1} \int_{\R^d}}

%--- Integral over all \R as -\iny to \iny
\newcommand{\IntAll}
    {\int_{-\iny}^\iny}

%--- Schwartz functions
\newcommand{\Schwartz}
    {\ensuremath \Cal{S}}

%--- Schwartz functions over R
\newcommand{\SchwartzR}
    {\ensuremath \Schwartz (\R)}

%--- Schwartz functions over R^d
\newcommand{\SchwartzRd}
    {\ensuremath \Schwartz (\R^d)}

%--- Schwartz-dual functions
\newcommand{\SchwartzDual}
    {\ensuremath \Cal{S}'}

%--- Schwartz-dual functions over R
\newcommand{\SchwartzRDual}
    {\ensuremath \Schwartz' (\R)}

%--- Schwartz-dual functions over R^d
\newcommand{\SchwartzRdDual}
    {\ensuremath \Schwartz' (\R^d)}

%--- Strange tilde-norm norm on H^s ---
\newcommand{\HSNorm}[1]
    {\norm{#1}_{H^s(\R^2)}}

%--- Strange tilde-norm norm on H^#2 ---
\newcommand{\HSNormA}[2]
    {\norm{#1}_{H^{#2}(\R^2)}}

%--- Holder, with an umlaut over the "o"
\newcommand{\Holder}
    {H\"{o}lder }

%--- Holder's, with an umlaut over the "o"
\newcommand{\Holders}
    {H\"{o}lder's }

%--- Holder, with an umlaut over the "o"
\newcommand{\Holderian}
    {H\"{o}lderian }

%--- Frechet, with an accent 
\newcommand{\Frechet}
	{Fr\'{e}chet }

%--- Strange tilde-norm (norm on H^s ---
\newcommand{\HolderRNorm}[1]
    {\widetilde{\Vert}{#1}\Vert_r}

%--- L-infinity norm ---
\newcommand{\LInfNorm}[1]
    {\norm{#1}_{L^\iny(\Omega)}}

%--- L-infinity norm, small size ---
\newcommand{\SmallLInfNorm}[1]
    {\smallnorm{#1}_{L^\iny}}

%--- L-1 norm ---
\newcommand{\LOneNorm}[1]
    {\norm{#1}_{L^1}}

%--- Small L-1 norm ---
\newcommand{\SmallLOneNorm}[1]
    {\smallnorm{#1}_{L^1}}

%--- L-2 norm ---
\newcommand{\LTwoNorm}[1]
    {\norm{#1}_{L^2(\Omega)}}

%--- Small L-2 norm ---
\newcommand{\SmallLTwoNorm}[1]
    {\smallnorm{#1}_{L^2}}

%--- L-p norm ---
\newcommand{\LpNorm}[2]
    {\norm{#1}_{L^{#2}}}

%--- Small L-p norm ---
\newcommand{\SmallLpNorm}[2]
    {\smallnorm{#1}_{L^{#2}}}

%--- l-1 norm ---
\newcommand{\lOneNorm}[1]
    {\norm{#1}_{l^1}}

%--- l-2 norm ---
\newcommand{\lTwoNorm}[1]
    {\norm{#1}_{l^2}}

\newcommand{\MsrNorm}[1]
    {\norm{#1}_{\Cal{M}}}

%--- Fourier transform when argument is too long for \widehat
\newcommand{\FTF}
    {\Cal{F}}

%--- Inverse Fourier transform
\newcommand{\FTR}
    {\Cal{F}^{-1}}

\newcommand{\InvLaplacian}
    {\ensuremath{\widetilde{\Delta}^{-1}}}

\newcommand{\EqDef}
    {\hspace{0.2em}={\hspace{-1.2em}\raisebox{1.2ex}{\scriptsize def}}\hspace{0.2em}}
    
\newcommand{\TR}{\mathbf{P}_{V_R}}
\newcommand{\UR}{\mathbf{U}_R}
\newcommand{\PHR}{\mathbf{P}_{H_R}}
\newcommand{\YR}{\ol{\mathbf{U}}_R}
\newcommand{\TRinv}{\mathbf{P}_{V_R}^{-1}}

%--- The function spaces we define
\newcommand{\X}{\ensuremath{\BB{E}}} 
\newcommand{\Y}{\ensuremath{\BB{Y}}} 
\newcommand{\Z}{\ensuremath{{\BB{E}^1}}}

%
%--------------------- Commands specific to this paper -----------------
%

%
%-------------------- Redefined from preamble ---------------------------
%

\title
    [Infinite-energy 2D statistical solutions in fluid mechanics]
    {Infinite-energy 2D statistical solutions to the equations of incompressible fluids}

%--- Information for first author
\author{James P. Kelliher}
% --- Address of record for the research reported here
\address{Department of Mathematics, University of California, Riverside, 900 University Ave.,
Riverside, CA 92521}
%--- Current address
\curraddr{Department of Mathematics, University of California, Riverside, 900 University Ave.,
Riverside, CA 92521}
\email{kelliher@math.ucr.edu}
% \thanks{}

%    General info
 % ; Secondary }

\date{} %  and, in revised form, .}

% \dedicatory{}


\begin{abstract}

We develop the concept of an infinite-energy statistical solution to the Navier-Stokes and Euler equations in the whole plane. We use a velocity formulation with enough generality to encompass initial velocities having bounded vorticity, which includes the important special case of vortex patch initial data. Our approach is to use well-studied properties of statistical solutions in a ball of radius $R$ to construct, in the limit as $R$ goes to infinity, an infinite-energy solution to the Navier-Stokes equations. We then construct an infinite-energy statistical solution to the Euler equations by making a vanishing viscosity argument.

\end{abstract}

\maketitle

%--- Don't include this for arXiv submission, as it gives a timestamp itself
\IfarXivElse{}{
		  \DetailForSubmission{
                 Compiled on \textit{\textbf{\today}}
                 }
             }

%--- Give date and time and detail level UNLESS for submission
\DetailMarginNote{
    \begin{small}
        \begin{flushright}
            Compiled on \textit{\textbf{\today}}

            \ExplainDetailLevel

        \end{flushright}
    \end{small}
    }

%-----------table of contents----------------------
\tableofcontents

%
% Section
%
\section{Introduction}\label{S:Introduction}

\Ignore{ % Ignore
\noindent  A 2D statistical solution to the Navier-Stokes (SSNS) or Euler equations (SSE) is a Borel probability measure on  a function space representing velocity or some other parameter such as vorticity (scalar curl of the velocity), with the measure evolving over time in accordance with the respective set of fluid equations. One typically works with a weak form of the solution that is applicable to higher dimensions as well.
} % End Ignore

We develop the concept of a statistical solution to the Navier-Stokes (SSNS) or Euler equations (SSE) in the plane for an important class of velocity fields having sufficient decay of the vorticity at infinity to recover uniquely the velocity field from the vorticity. In particular, this class of velocity fields includes the important case of a vortex patch: a velocity field whose initial vorticity is the characteristic function of a bounded domain.

Our starting point is the velocity formulation of a SSNS on a bounded domain given by Foias in \cite{Foias1972}. (A highly accessible account of the theory of SSNSs is given in \cite{FMRT}, to which we refer often.)
We adapt this formulation slightly, of necessity changing the energy equality and using the same class of test functions as for homogeneous solutions in the plane (\refS{SSNS}). Our definition of an infinite-energy SSE is the same with the viscosity set to zero. We construct our infinite-energy SSNS by showing that it is the limit, in a special sense, of a sequence of statistical solutions on balls of radius $R$ as $R \to \iny$. At the core of our approach is the expanding domain limit for deterministic solutions to the Navier-Stokes equations established in \cite{K2005ExpandingDomain}, extended to handle infinite-energy solutions.
We then construct our infinite-energy SSE by making a vanishing viscosity argument.

\Ignore{ % This was the original opening
\noindent The equations of incompressible fluid mechanics in the whole plane, including the Navier-Stokes and Euler equations, can be made sense of for velocity fields that have infinite energy. The manner in which this is done depends on how ``infinite'' the energy is. It is possible to interpret solutions even in the situation where the velocity and its associated vorticity (scalar curl) are only in $L^\iny$ (\cite{GIM1999}, \cite{GMS2001}, \cite{Serfati1995A}, \cite{C2008}). We restrict our attention here to an important class of velocity fields for which there is sufficient decay of the vorticity at infinity to recover uniquely the velocity field from the vorticity. (We make this more precise below.) In particular, this class of velocity fields includes the important case of a vortex patch: a velocity field whose initial vorticity is the characteristic function of a bounded domain.

There are two basic approaches in the literature to the analysis of deterministic solutions to the Navier-Stokes or Euler equations for the infinite energy velocity fields that we consider, one based on the velocity formulation of the equations the other based on the vorticity formulation.

In the velocity formulation, the Navier-Stokes or Euler equations are written solely in terms of the velocity.
One considers an initial velocity lying in the space $E_m$
of \cite{C1996} and \cite{C1998}.  A vector $v$ belongs to $E_m$ if it is divergence-free
and can be written in the form $v = \sigma + v'$, where $v'$ is in
$L^2(\R^2)$ and where $\sigma$ is a smooth stationary solution to the Euler equations whose vorticity is radially symmetric and compactly supported. (See \refS{FS} for more details.) A unique global solution to the Navier-Stokes equations exists and remains in the space $E_m$ for all time as long as the forcing has finite energy. The same can be said of solutions to the Euler equations if one imposes restrictions on the initial vorticity; for instance, that it lie in $L^1 \cap L^\iny$.

In the vorticity formulation, the equations are written in terms of the vorticity, with the velocity recovered from the vorticity using the Biot-Savart law (\refL{BSLaw}). To apply the Biot-Savart law, however, one must have sufficient regularity and decay of the vorticity, so even for solutions to the Navier-Stokes equations, we must have higher regularity of the initial velocity. We need such regularity to have well-posedness even in the velocity formulation of the Euler equations, so this is not a limitation there. When the required properties of the velocity are not so critical, this is a very effective procedure, but the velocity will not, in general, lie in $E_m$ for any $m$. An advantage of this approach is that one never directly has to face the lack of finite energy. A difficulty with the approach is that controlling the velocity can be complicated.

\Ignore{ % If I start trying to the list papers, things will get carried away, and I will miss an important example.
In relation to weak solutions to the Navier-Stokes or Euler equations with initial velocity in functions classes that are near to (or beyond) the limit of what is known to lead to unique solutions to the Euler equations---the situation that concerns us here---there are many examples of the velocity formulation approach: \cite{C1996} and \cite{C1998} mentioned above are prototypical. The purely vorticity formulation approach in this context is adopted \cite{VIncompressible} in proving uniqueness of solutions to the Euler Equations and in \cite{BenArtzi1994} ......
} % End Ignore

Analogous infinite-energy statistical solutions to the Navier-Stokes and Euler equations seem not to have been directly considered in the literature. In this paper, we show that, in fact, a satisfactory infinite-energy theory for the velocity formulation can be established by considering the limit, in a special sense, of statistical solutions to these equations in a ball of radius $R$ as $R$ increases to infinity.
} % End Ignore

\Ignore{ % Ignore
The Navier-Stokes equations describe the behavior of a classical incompressible fluid with constant positive viscosity; the zero-viscosity (inviscid) case, leads to the Euler equations. We say that the vanishing viscosity limit holds if the solutions to the Navier-Stokes equations converge in an appropriate sense over a finite time interval to a solution to the Euler equations in the limit as the viscosity goes to zero. For deterministic solutions to the equations of incompressible fluids, the vanishing viscosity limit has been established in a variety of settings, including:
\begin{enumerate}
\item
All of $\R^d$, $d \ge 2$ for smooth solutions up to the time for which the solution to the Euler equations remains smooth.

\item

All of $\R^2$, when the initial vorticity is in $L^{p_0} \cap L^\iny$ for some $p_0 \le 2$ and where the energy is allowed to be infinite (the velocity lying in the space $E_m$, described below). (\cite{C1996}, with extension to unbounded vorticities in \cite{K2003}.)

\item

A bounded domain with Navier friction boundary conditions for smooth solutions in $\R^3$ (\cite{IP2006}) or in $\R^2$ for initial velocities as in (2) above (\cite{CMR}, \cite{FLP}, \cite{KNavier}).

\item

Special geometries or assumptions on the initial velocity that lead to improved bounds on the rates of convergence.
\end{enumerate}

Perhaps the most important example of the last type is the case of a 2D vortex patch, considered in \cite{CW1995}, \cite{CW1996}, \cite{AD2004}.  A 2D vortex patch is a velocity field $u$ whose vorticity (scalar curl) $\omega$ is bounded and has compact support. In the classical case one considers the characteristic function of a bounded domain with some degree of smoothness to its boundary. This is a special case of that considered in (2), but where the topic of interest becomes upper and lower bounds on the rate of convergence, as well as the strength of the convergence (showing convergence of the vorticity, for instance, and not just the velocity).
 
A 2D statistical solution to the Navier-Stokes or Euler equations, in the velocity formulation, is a Borel probability measure on  a function space representing velocity, with the measure evolving over time in accordance with the respective set of fluid equations. Vanishing viscosity arguments in the deterministic setting extend, in a fairly straightforward way, to statistical solutions to the Navier-Stokes and Euler equations, as long as a suitable existence theory for such solutions in the given setting can be established. (A specific example that illustrates this fact is given in the proof of \refT{FMRTForEInf}.) Rates of convergence, however, are not as naturally defined for statistical solutions. It is our main purpose here to establish an existence (and in the case of the Navier-Stokes equations, uniqueness) theory appropriate for application to vortex patch initial data, and more generally, to bounded initial vorticity. The main consideration is allowing the velocities to have infinite energy.
} % End Ignore

In the deterministic setting, the infinite-energy solutions that we consider correspond to an initial velocity lying in the space $E_m$ of \cite{C1996} and \cite{C1998}.  A vector $v$ belongs to $E_m$ if it is divergence-free
and can be written in the form $v = \sigma + v'$, where $v'$ is in
$L^2(\R^2)$ and where $\sigma$ is a smooth stationary solution to the Euler equations whose vorticity is radially symmetric and compactly supported. (See \refS{FS} for more details.) A unique global solution to the Navier-Stokes equations exists and remains in the space $E_m$ for all time as long as the forcing has finite energy. The same can be said of solutions to the Euler equations if one imposes restrictions on the initial vorticity; for instance, that it lie in $L^1 \cap L^\iny$. This encompasses the case of a classical vortex patch---initial vorticity that equals the characteristic function of a bounded domain. (See also \refC{CorBSLaw}.)

A key parameter of any vortex patch is the total mass of its vorticity,
\begin{align*}
	m = \int_{\R^2} \omega.
\end{align*}
Only when $m = 0$ will the velocity field have finite energy (lie in $L^2$), which excludes the case of a classical vortex patch. In a sense, $m$ measures how infinite the energy is.

\Ignore{ % This is redundant
In the deterministic setting, one handles this situation by extending slightly the theory of finite-energy solutions to the Navier-Stokes and Euler equations to allow infinite energy solutions. It turns out, that given initial vorticity with total mass $m$, the vorticity for the solutions to these equations will always have total mass $m$, as long as the forcing has finite energy. This greatly facilitates working with such solutions, since they lie for all time in the space $E_m$, which has one, definite ``infiniteness'' of energy.
} % ENd Ignore

When working with statistical solutions one would like to allow $m$ to take on different values, because classical vortex patches that are nearly identical will typically have different values of $m$. Thus, we need to consider the spaces $E_m$ for all values of $m$ simultaneously.

\Ignore{ % Ignore

We will extend (in 2D only) the theory of statistical solutions developed by Foias in \cite{Foias1972} to apply to infinite-energy solutions in the whole plane. Our approach is to use  the existence theory for statistical solutions to Navier-Stokes equations for a bounded domain in the plane from \cite{Foias1972}, and let the size of the domain increase to infinity, showing that in the limit, in a special sense, one obtains a statistical solution to the Navier-Stokes equations in the whole plane that allows for infinite energy. At the core of our approach is the expanding domain limit for deterministic solutions to the Navier-Stokes equations established in \cite{K2005ExpandingDomain}, extended to handle infinite-energy solutions.

We construct an infinite-energy statistical solution to the Euler equations by showing that statistical solutions to the Navier-Stokes equations converge to such a solution over any finite time interval as the viscosity vanishes. (Another type of vanishing viscosity limit looks at what happens to long term time-averaged solutions to the Navier-Stokes equations as the viscosity vanishes. This kind of limit has more meaning in the theory of turbulence. For instance, in \cite{CR2007}, Constantin and Ramos  establish a long time-averaged vanishing viscosity limit for damped and driven stationary statistical solutions to the Navier-Stokes equations to their inviscid form.)
} % End Ignore

We give a velocity rather than a vorticity formulation of our statistical solutions for several reasons. First, a vorticity formulation would require imposing higher regularity on the initial vorticity than required for solutions to the Navier-Stokes equations: it would be technically quite difficult to assume anything weaker than the initial vorticity lying in $L^1$, as in \cite{BenArtzi1994}. Second, it would be hard to obtain convergence of the vorticity in the vanishing viscosity limit, even with higher regularity of the initial data, without knowing that the velocity decays at infinity, and this does not come from the Biot-Savart law. Third, a start in this direction has already been made in \cite{CR2007} for time-independent solutions to damped and driven Navier-Stokes and Euler equations in the vorticity formulation. 

Constantin and Ramos do not specifically address infinite-energy solutions in \cite{CR2007}; however, their definition of a such a solution requires no change at all to encompass infinite-energy solutions and neither does their proof of the vanishing viscosity limit.
\ifthenelse{\value{bIncludeConstantinRamosSection}=0}
{}
{(See \refR{InfiniteFiniteEnergyCR}.)}
Their construction of a stationary statistical solution to the Navier-Stokes equations as a long-time average of a deterministic solution to the damped and driven Navier-Stokes equations does assume finite energy. Nothing deep need be done, however, to extend their construction to allow infinite energy solutions (and so allow initial vortex patch data): one need only assume infinite-energy forcing. 
\ifthenelse{\value{bIncludeConstantinRamosSection}=0}
{}
{We outline the approach briefly in \refS{DD}.}

This paper is organized as follows: In \refS{FS}, we define the function spaces in which we will work. We characterize the projection operator we will use to construct the initial velocities in $\Omega_R$ in \refS{Projection}. In \refS{WeakSolutions}, we define weak deterministic solutions to the Navier-Stokes and Euler equations and give the basic well-posedness and regularity results for such solutions. The deterministic expanding domain limit of \cite{K2005ExpandingDomain} is established for infinite-energy solutions in \refS{Expanding}. We give the definition of a statistical solution to the Navier-Stokes equations in velocity form, for finite as well as infinite energy, in \refS{SSNS}, and construct an infinite-energy statistical solution to the Navier-Stokes equations in \refS{ConstructNS}, showing that it is unique. In \refS{ConstructE} we construct an infinite-energy statistical solution to the Euler equations using a vanishing viscosity argument.
\ifthenelse{\value{bIncludeConstantinRamosSection}=0}
{}
{Finally, in \refS{DD}, we discuss the damped and driven solutions to the Navier-Stokes equations of \cite{CR2007} with infinite energy.}

%
% Section
%
\section{Function spaces and the Biot-Savart Law}\label{S:FS}

\noindent Let
\begin{align*}
	\Omega_R = \text{the disk of radius } R \text{ centered at the origin},
\end{align*}
with $\Omega_\iny = \R^2$, and define the classical function spaces of incompressible fluid mechanics,
\begin{align*}
	H_R &= H(\Omega_R)
		= \set{u \in (L^2(\Omega_R))^2 \colon \dv u = 0, \, u \cdot \mathbf{n} = 0 \text{ on } \prt \Omega_R}, \\
	V_R &= V(\Omega_R)
		= \set{u \in (H^1(\Omega_R))^2 \colon \dv u = 0, \, u = 0 \text{ on } \prt \Omega_R}, \\
	H &= H_\iny = H(\R^2), \quad
	V = V_\iny = V(\R^2).
\end{align*}

We endow $H_R$ with the $L^2$-norm. For $V_R$, we use the $H^1$-norm:
\begin{align}\label{e:VNorm}
	\norm{u}_{V_R} = \norm{u}_{L^2(\Omega_R)} + \norm{\grad u}_{L^2(\Omega_R)}.
\end{align}
Note, in particular, that
\begin{align}\label{e:H1Bounds}
	\norm{u}_{L^2(\Omega_R)} \le \norm{u}_{V_R}, \qquad
	\norm{\grad u}_{L^2(\Omega_R)} \le \norm{u}_{V_R}.
\end{align}
Had we used the Poincare inequality to replace \refE{VNorm} with the equivalent norm that includes only the second term, as is normally done for a bounded domain, it would have introduced a factor of $R$ in the right-hand side of the first inequality, preventing us from having a consistent norm with which to compare solutions on $\Omega_R$ for different values of $R$.

Our deterministic infinite energy solutions will lie in the space $E_m$ of \cite{C1998}.  A vector $v$ belongs to $E_m$ if it is divergence-free
and can be written in the form $v = \sigma + v'$, where $v'$ is in
$L^2(\R^2)$ and where $\sigma$ is a \textit{stationary vector field},
meaning that $\sigma$ is of the form,
\begin{align}\label{e:Stationary}
    \sigma = \pr{-\frac{x^2}{r^2} \int_0^r \rho g(\rho) \, d \rho,
                \;\frac{x^1}{r^2} \int_0^r \rho g(\rho) \, d \rho}
\end{align}
with $g$ in $C_0^\iny(\R)$. The subscript $m$ is the integral over
all space of the vorticity, 
\begin{align*}
	\omega(v) =  \prt_1 v^2 - \prt_2 v^1.
\end{align*}

$E_m$ is an affine space; fixing an origin, $\sigma$, in $E_m$ we can
define a ``norm'' by $\norm{\sigma + v'}_{E_m} = \LTwoNorm{v'}$.
Convergence in $E_m$ is equivalent to convergence in the $L^2$--norm to
a vector in $E_m$.

We will find it convenient to fix a choice of origin for $E_m$ as follows. For $E_1$ we choose $\sigma_1$ of the form \refE{Stationary} with $\omega(\sigma_1)$ supported in the unit disk and with
\begin{align*}
	\int_{\Omega_1} \omega(\sigma_1)
		= \int_{\R^2} \omega(\sigma_1)
		= 1.
\end{align*}
We can then use $\sigma_m = m \sigma_1$ as an origin for $E_m$.

\Ignore{ % Only needed in an ignored section
A simple integration shows that
\begin{align}\label{e:sigma1Norm}
	\norm{\sigma_1}_{L^2(\Omega_R)}^2
		= C_1 +  \frac{1}{2 \pi} \log R,
\end{align}
for all $R \ge 1$.
}

Let $\psi_{\sigma_1}$ be a given fixed stream function for $\sigma_1$. As in in \cite{KLL2007}, $\psi_{\sigma_1}$ is radially symmetric with
\begin{align}\label{e:psiForrLarge}
	\psi_{\sigma_1}(x) = C_2 +  \frac{1}{2 \pi} \log \abs{x}
\end{align}
for all $\abs{x} \ge 1$. For $\abs{x} \ge 1$, $\abs{\sigma_1(x)} = 1/\abs{x}$ by \refE{Stationary}, and
\begin{align}\label{e:sigma1RR1}
	\begin{split}
	&\norm{\sigma_1}_{H^1(\Omega_R \setminus \Omega_{R - 1})}^2
		= 2 \pi \int_{R-1}^R \frac{1}{r^2} r \, dr + 2 \pi \int_{R-1}^R \frac{1}{r^4} r \, dr \\
		&\qquad
			= 2 \pi \log(R/(R - 1)) + \pi\brac{(R - 1)^{-2} - R^{-2}}
			\to 0 \text{ as } R \to \iny.
	\end{split}
\end{align}
\refE{psiForrLarge} also gives $\Delta \sigma_1 = \Delta \grad^\perp \psi_{\sigma_1} =  \grad^\perp \Delta \psi_{\sigma_1} = 0$ on $\Omega_1^C$.

The spaces $H_R$ and $E_m$, or $V_R$ and $E_m \cap \dot{H}^1(\R^2)$, are the appropriate ones for initial velocities for weak deterministic solutions to the Navier-Stokes equtions, but for the Euler equations more regularity is required to obtain well-posedness. Rather than being as general as possible, we will assume that for deterministic solutions the initial vorticity lies in $L^\iny$ for solutions on $\Omega_R$ and in $L^{p_0} \cap L^\iny$, for some $p_0 < 2$ for solutions on $\R^2$. Slightly unbounded vorticities could be handled, as in \cite{K2005ExpandingDomain}, with little complication. This gives not only existence but uniqueness of the solutions. (The uniqueness of solutions for bounded initial vorticity is due to Yudovich \cite{Y1963}, as is the uniqueness for unbounded vorticities \cite{Y1995}.)

Thus, we fix $p_0 < 2$, and define the spaces
\begin{align*}
	\Y_m = \set{u \in E_m \colon \omega(u) \in L^{p_0} \cap L^\iny},
\end{align*}
with ``norm''
\begin{align*}
	\norm{u}_{\Y_m}
		= \norm{u}_{E_m} + \norm{\omega(u - \sigma_m)}_{L^{p_0} \cap L^\iny}
\end{align*}
and
\begin{align*}
	\Y(\Omega_R)
		= \set{u \in H(\Omega_R) \cap H^1(\Omega_R) \colon \omega(u) \in L^\iny}
\end{align*}
with norm
\begin{align*}
	\norm{u}_{\Y(\Omega_R)}
		= \norm{u}_{H^1(\Omega_R)} + \norm{\omega(u)}_{L^{p_0} \cap L^\iny(\Omega_R)}, \quad R < \iny.
\end{align*}
Because $L^{p_0}(\Omega_R) \subseteq L^\iny(\Omega_R)$, using only the $L^\iny$-norm of $\omega(u)$ in the $Y(\Omega_R)$-norm would give a simpler, equivalent norm. We avoid doing this, however, for the same reason we avoided the use of Poincare's inequality in defining the $V_R$-norm in \refE{VNorm}.

For statistical solutions in the whole plane, we do not want to assume that the value of $m$ is fixed, so we must deal with larger spaces. For statistical solutions to the Navier-Stokes equations we will use
\begin{align*}
	\X = \bigcup_{m \in \R} E_m
		\text{ and }
		\Z = \bigcup_{m \in \R} E_m \cap \dot{H}^1(\R^2)
\end{align*}
and for statistical solutions to the Euler equations we will use
\begin{align}\label{e:Y}
	\Y = \bigcup_{m \in R} \Y_m.
\end{align}
 $\X$, $\Y$, and $\Z$ are function spaces, being closed under addition, with the norms
 \begin{align*}
 	\norm{\sigma_m + v}_\X &= \abs{m} + \norm{v}_{L_2}, \quad
	\norm{u}_\Y = \norm{u}_\X +  \norm{\omega(u)}_{L^{p_0} \cap L^\iny}, \\
	\norm{\sigma_m + v}_\Z &= \norm{\sigma_m + v}_\X + \norm{\grad v}_{L^2(\R^2)}.
\end{align*}
These norms induce metrics on their respective spaces. Because $H$ and $V$ are separable, so too are $\X$ and $\Z$. The space $\Y$, however, is not separable, because $L^\iny(\R^2)$ is not.

There exists a unique decomposition of any $u$ in $\X$, $\Y$, or $\Z$ of the form $u = \sigma_m + v$, $m \in \R$, $v \in H$. Given such a $u$, we define
\begin{align}\label{e:L2Part}
	m(u) = m, \quad \sigma(u) = \sigma_{m(u)}.
\end{align}

\begin{definition}\label{D:BoundedSupport}
We say that the support of a measure $\mu$ on the function space $X$ is \textit{bounded in $X$} if
\begin{align*} % \label{e:muBoundedX}
	\supp \mu \subseteq \set{u \in X \colon \norm{u}_X \le M}
		\text{ for some } M < \iny.
\end{align*}
If $Y$ is a subspace of $X$, we say that the support of a measure $\mu$ is \textit{$(X, Y)$-bounded} if
\begin{align*} % \label{e:muBoundedXY}
	\supp \mu \subseteq \set{u \in Y \colon \norm{u}_X \le M}
		\text{ for some } M < \iny.
\end{align*}
That is, the support of $\mu$ lies in the subspace, but only its norm in the full space is controlled.
\end{definition}

\begin{lemma}\label{L:BSLaw}[Biot-Savart law]
    Let $p$ be in $[1, 2)$ and let $q > 2p/(2 - p)$. For any vorticity $\omega$ in
    $L^p(\R^2)$ there exists a unique divergence-free vector field $u$ in
    $L^p(\R^2) + L^q(\R^2)$ whose curl is $\omega$, with $u$ being given by the
    Biot-Savart law,
    \begin{align}\label{e:BSLaw}
        u = K * \omega.
    \end{align}
    Here, $K$ is the Biot-Savart kernel, $K(x) = (1/2 \pi) x^\perp/\abs{x}^2$.
\end{lemma}
\begin{proof}
    See the proof of Proposition 3.1.1 p. 44-45 of \cite{C1998}.
\end{proof}

\begin{cor}\label{C:CorBSLaw}
For any vorticity $\omega$ in $L^1 \cap L^\iny(\R^2)$ there exists a unique divergence-free vector field $u$ in $L^\iny(\R^2)$ whose curl is $\omega$, with $u$ given by \refE{BSLaw}. If $\omega$ is also compactly supported then $u$ lies in $E_m$, where $m = \int_{\R^2} \omega$.
\end{cor}
\begin{proof}
By \refL{BSLaw} applied with $p = 1$, $K*\omega$ is the unique vector field in $L^1 \cap L^\iny$ whose vorticity is $\omega$. But also,
\begin{align*}
	\norm{K*\omega}_{L^\iny}
		&\le \norm{(\chi_{\Omega_1} K) *\omega}_{L^\iny}
		+ \norm{(1 - \chi_{\Omega_1}) K *\omega}_{L^\iny} \\
		&\le \norm{\chi_{\Omega_1} K}_{L^1} \norm{\omega}_{L^\iny}
		+ \norm{(1 - \chi_{\Omega_1}) K}_{L^\iny}  \norm{\omega}_{L^1} \\
		&\le C \norm{\omega}_{L^1 \cap L^\iny}.
\end{align*}
Here, $\chi_A$ is the characteristic function of $A$. This shows that, in fact, $K*\omega$ is in $L^\iny$ and is the unique such vector field. The last statement in the corollary follows from Lemma 1.3.1 of \cite{C1998}.
\end{proof}

%
% Section
%
\section{Projection operators}\label{S:Projection}

\noindent Let $\TR \colon \Z \to V_R$ be restriction to $\Omega_R$ followed by projection onto $V_R$. $\TR$ is well-defined because as Hilbert spaces $V_R$ is a closed subspace of $H^1_{div}(\Omega_R)$, the space of all divergence-free vector fields in $(H^1(\Omega_R))^2$, endowed with the inner product,
\begin{align*}
	\innp{u, v}_{H^1_{div}(\Omega_R)}
		= (u, v) + (\grad u, \grad v).
\end{align*}
We can describe $\TR$ explicitly by characterizing $V_R^\perp$, the orthogonal complement of $V_R$ in $H^1_{div}(\Omega_R)$. By definition, $w$ is in $V_R^\perp$ if and only if $\innp{w, v}_{H^1_{div}(\Omega_R)} = 0$ for all $v$ in $V_R$. Treating $\Delta w$ as a distribution, integrating by parts gives $(w - \Delta w, v) = 0$ for all $v$ in $\Cal{V}_R$, where $\Cal{V}_R = V_R \cap \Cal{D}(\Omega_R)$. It follows from this that $w$ is in $V_R^\perp$ if and only if
\begin{align}\label{e:wVR}
	\Delta w - w = \grad p
\end{align}
for some $p$ in $L^2(\Omega)$ (see, for instance, Proposition I.1.1 of \cite{T2001}) and, of course, $\dv w = 0$.

Now let $u$ lie in $H^1_{div}(\Omega_R)$ and let $\ol{u} = \TR u$. Then $w = u - \ol{u}$ lies in $V_R^\perp$ so from \refE{wVR},
\begin{align}\label{e:PVR}
	\left\{
	\begin{array}{ll}
		\Delta \ol{u} - \ol{u} = \Delta u - u + \grad p
				& \text{in } \Omega_R, \\
		\dv \ol{u} = \Delta p = 0
		     & \text{in } \Omega_R, \\
		\ol{u} = 0
		     & \text{on } \prt \Omega_R. \\
	\end{array}
	\right.
\end{align}
Equality is to hold in a weak sense in \refE{PVR}. Since $u$ is in $H^1$, however, $f = \Delta u - u + \ol{u}$ is in $H^{-1}$, which is sufficient to conclude that $\ol{u}$ is in $V$ and $p$ is in $L^2$. (See, for instance, Remark I.2.6 of \cite{T2001}.) Also, the solution to \refE{PVR} is unique because otherwise it would follow that $-1$ is an eigenvalue of the Stokes operator, $- \Cal{P} \Delta$, where $\Cal{P}$ is the Leray projector. But the Stokes operator is positive-definite, so all its eigenvalues are positive.

The estimates involving the operator $\TR$ are hard to prove directly using this characterization. It is simpler to employ an approximate projection operator $\UR$, and use the fact that projection into $V_R$ gives the closest vector field in $V_R$ to the vector field being projected. (The same idea is used for projection into $H_R$ in \cite{KLL2007, IK2007}.)

To define $\UR$ we need two cutoff functions, $\varphi_R$ and $h_R$.

Let $\varphi_1$ in $C^\iny(\Omega_1)$ take values in $[0, 1]$ and be defined so that $\varphi_1 = 1$ on $\Omega_{1/2}$ and so that $\varphi_1$ and $\grad \varphi_1$ are both zero on $\prt \Omega_R$. Let $\varphi_R(\cdot) = \varphi_1(\cdot/R)$. Observe that $\varphi_R$ and $\grad \varphi_R$ both vanish on $\prt \Omega_R$.

Let $g$ in $C^\iny([0, 3/4])$ taking values in $[0, 1]$ be defined so that
$g(0) = g'(0) = 0$ and $g = 1$ on $[1/2, 1]$. Then define $h_R$
in $C^\iny(\Omega_R)$ by $h_R(x) = g(R - \abs{x})$ for points $x$ in $\Omega_R \setminus \Omega_{R - 1}$
and $h_R = 1$ on $\Omega_{R - 1}$. Observe that
\begin{align}\label{e:hRBounds}
	\norm{h_R}_{C^k} \le C_k,
\end{align}
$k = 0, 1, \dots,$ for constants $C_k$ independent of $R$ in $[1, \iny)$. Also, $h_R = 0$ and $\grad h_R = 0$ on $\prt \Omega_R$.

\begin{definition}\label{D:ProjVRApprox}
Define $\UR \colon \Z \to V_R$ by
\begin{align*}
	\UR (u) = \grad^\perp (h_R (\psi_{\sigma_m} - \psi_{\sigma_m}(R)))
			+ \grad^\perp(\varphi_R \psi_v)
\end{align*}
for $u = \sigma_m + v$ in $E_m$. Here, $\psi_v$ is the stream function for $v$ chosen so that $\int_{\Omega_R} \psi_v = 0$ on $\prt \Omega_R$.
\end{definition}

\begin{lemma}\label{L:ProjVR}
	$\TR$ maps $\Z$ continuously onto $V_R$ with
	\begin{align}\label{e:TRNSHighLimit}
		\norm{u - \TR u}_{H^1(\Omega_R)}
			\le C \norm{u - \sigma(u)}_{H^1(\Omega_R \setminus \Omega_{R/2})} + \abs{m(u)} \beta(R),
	\end{align}
	where
	\begin{align*}
		\beta(R)
			= \norm{\sigma_1}_{H^1(\Omega_R \setminus \Omega_{R - 1})}
			\to 0 \text{ as } R \to \iny,
	\end{align*}
	and
	\begin{align}\label{e:TRNSHighInequality}
		\norm{\TR (u - \sigma_m)}_{V_R} \le \norm{u - \sigma_m}_V,
	\end{align}
	\begin{align}\label{e:TRNSHighBound}
		\norm{\TR u}_{V_R} \le \norm{u}_{\Z} + C \abs{m(u)}.
	\end{align}
\end{lemma}
\begin{proof}
That $\TR$ maps onto $V_R$ is clear, and it is continuous because the restriction and the projection operators are continuous. 

To prove \refE{TRNSHighLimit}, let $u = \sigma_m + v$ in $E_m \cap \dot{H}^1(\R^2)$. Then
\begin{align}\label{e:uURuBound}
	\norm{u - \UR u}_{H^1(\Omega_R)}
		\le \norm{\sigma_m - \UR \sigma_m}_{H^1(\Omega_R)} + \norm{v - \UR v}_{H^1(\Omega_R)}.
\end{align}
It  follows from the proof of Lemma 4.2 of \cite{K2005ExpandingDomain} that
\begin{align*}
	\norm{v - \UR v}_{H^1(\Omega_R)}
		\le C \norm{u - \sigma(u)}_{H^1(\Omega_R \setminus \Omega_{R/2})}.
\end{align*}
Also, letting $\psi = \psi_{\sigma_m} - \psi_{\sigma_m}(R)$ and using \refE{hRBounds},
\begin{align*}
	&\norm{\sigma_m - \UR \sigma_m}_{H^1(\Omega_R)}
		= \smallnorm{\grad^\perp \psi
			- \grad^\perp (h_R \psi)}_{H^1(\Omega_R)} \\
	&\qquad \le \norm{(1 - h_R) \sigma_m}_{H^1(\Omega_R)}
			+ \smallnorm{\grad^\perp h_R \psi}_{H^1(\Omega_R)} \\
	&\qquad \le \norm{1 - h_R}_{C^1}  \norm{\sigma_m}_{H^1(\Omega_R \setminus \Omega_{R - 1})}
		+ \norm{\grad h_R}_{C^1} \norm{\psi}_{H^1(\Omega_R \setminus \Omega_{R - 1})} \\
	&\qquad \le C \norm{\sigma_m}_{H^1(\Omega_R \setminus \Omega_{R - 1})}
		+ C \norm{\psi}_{H^1(\Omega_R \setminus \Omega_{R - 1})} \\
	&\qquad \le C \norm{\sigma_m}_{H^1(\Omega_R \setminus \Omega_{R - 1})}
		+ C \norm{\psi}_{L^2(\Omega_R \setminus \Omega_{R - 1})}.
\end{align*}
Because $\Omega_R \setminus \Omega_{R - 1}$ has width 1 and $\psi$ vanishes on its outer boundary, we can apply Poincare's inequality with a constant that is independent of $R$ to give
\begin{align*}
	 \norm{\psi}_{L^2(\Omega_R \setminus \Omega_{R - 1})}
	 	\le  C\norm{\grad \psi}_{L^2(\Omega_R \setminus \Omega_{R - 1})}
		= C\norm{\sigma_m}_{L^2(\Omega_R \setminus \Omega_{R - 1})}.
\end{align*}
Thus,
\begin{align*}
	\norm{\sigma_m - \UR \sigma_m}_{H^1(\Omega_R)}
		\le C \norm{\sigma_m}_{H^1(\Omega_R \setminus \Omega_{R - 1})}
		= C \abs{m} \norm{\sigma_1}_{H^1(\Omega_R \setminus \Omega_{R - 1})},
\end{align*}
which vanishes as $R \to \iny$ by \refE{sigma1RR1}. This gives \refE{TRNSHighLimit} for $\UR$. Projection into $V_R$ gives the closest element in $V_R$, so \refE{TRNSHighLimit} holds for $\TR$.

\refE{TRNSHighInequality} follows easily:
\begin{align*}
	\norm{\TR (u - \sigma_m)}_{V_R}
		\le \norm{u - \sigma_m}_{V_R}
		\le \norm{u - \sigma_m}_V,
\end{align*}
the first inequality holding simply because $\TR$ is an orthogonal projection operator.

To prove \refE{TRNSHighBound}, let $u = \sigma_m + v$ in $E_m \cap \dot{H}^1(\R^2)$. Then
\begin{align*}
	&\norm{\TR u}_{V_R}
		\le \norm{\TR \sigma_m}_{V_R} + \norm{\TR v}_{V_R}
		\le \norm{\sigma_m}_{V_R} + \norm{v}_{V_R} \\
		&\qquad
		\le \norm{\sigma_m}_V + \norm{v}_V
		= \abs{m} \norm{\sigma_1}_V + \norm{v}_V
		= C \abs{m}  + \norm{v}_V \\
		&\qquad
		\le \norm{u}_\Z + C \abs{m}.
\end{align*}
\end{proof}

\begin{lemma}\label{L:PVRsigmam}
	$\TR \sigma_m$ is a stationary solution to the Euler equations on $\Omega_R$.
\end{lemma}
\begin{proof}
	Let $\ol{\sigma}_m = \TR \sigma_m$. Since $\psi_{\sigma_m}$ and $\omega(\sigma_m)$
	are radially symmetric, so too must
	$\psi_{\ol{\sigma}_m}$ and $\omega(\ol{\sigma}_m)$ be. But then
	\begin{align*}
		\omega(\ol{\sigma}_m \cdot \grad \ol{\sigma}_m)
		     = \ol{\sigma}_m \cdot \grad \omega(\ol{\sigma}_m)
				= 0,
	\end{align*}
	and thus $\ol{\sigma}_m \cdot \grad \ol{\sigma}_m = \grad p$ for some scalar field $p$.
\end{proof}

\begin{remark}\label{R:TRVersusPR1}
\refEThrough{TRNSHighLimit}{TRNSHighBound} continue hold if the projection operator $\TR$ is replaced by the approximate projection operator $\UR$, though a constant factor is introduced on the right-hand sides of \refE{TRNSHighInequality} and \refE{TRNSHighBound}. 
 \refEThrough{TRNSHighLimit}{TRNSHighBound} also hold with $\TR$ replaced by $\UR$ and $H^1$ replaced by $L^2$. This gives control not only on the $H^1$-norm but individual control on the $L^2$-norm. See \refRAnd{TRVersusPR2}{TRVersusPR3}.
\end{remark}

We will use the operator $\TR$ in establishing the deterministic expanding domain limit in \refS{Expanding} and in constructing statistical solutions to ($NS$) in \refS{ConstructNS}. For solutions to the Euler equations, we will need the following approximate truncation operator of \refD{YR} (or we could use projection into $Y(\Omega_R)$).
\begin{definition}\label{D:YR}
Define $\YR \colon \Y \to \Y(\Omega_R)$ by
\begin{align*}
	\YR(\sigma_m + v) = \sigma_m|_{\Omega_R} + T_R v,
\end{align*}
where $T_R \colon \Y_0 \to \Y(\Omega_R)$ is the operator in Lemma 4.2 of \cite{K2005ExpandingDomain}.
\end{definition}

Because $\Omega_R$ is a disk, $\sigma_m|_{\Omega_R}$ is in $H_R$ and so also in $\Y(\Omega_R)$. This is why $\YR \colon \Y \to \Y(\Omega_R)$. It is also why if we define $\PHR$ to be the restriction to $\Omega_R$ followed by projection into $H_R$ that for all $u$ in $\X$,
\begin{align*}
	\PHR u = \sigma(u) + \PHR (u - \sigma(u)).
\end{align*}
It follows from Lemma 4.2 of \cite{K2005ExpandingDomain} that for all $u$ in $\X$,
\begin{align}\label{e:PHRConvergence}
	\norm{\PHR u - u}_{L^2(\Omega_R)} \to 0
		\text{ as } R \to \iny.
\end{align}
Actually, Lemma 4.2 of \cite{K2005ExpandingDomain} applies to an approximate projection operator into $H_R$, but we are using, as in in \cite{KLL2007, IK2007}, the fact that projection into $H_R$ gives the closest vector field in $H_R$ to the vector field being projected.

%
% Section
%
\section{Weak deterministic solutions}\label{S:WeakSolutions}

\noindent 

\begin{definition}[Weak Navier-Stokes Solution]\label{D:WeakSolutionNS}
    Given viscosity $\nu > 0$, initial velocity $u_0$ in $H_R$,
    and forcing $f$ in $L^2_{loc}([0, \iny), H_R)$,
    $u$ in
    $L^2([0, T]; V_R)$ with $\prt_t u$ in $L^2([0, T]; V_R')$ is a weak
    solution to the Navier-Stokes equations  on $\Omega_R$ if $u(0) = u_0$ and
    \begin{align*}
        \mathbf{(NS)}
        \qquad \int_{\Omega_R} \prt_t u \cdot v
            + \int_{\Omega_R} (u \cdot \grad u) \cdot v
            + \nu \int_{\Omega_R} \grad u \cdot \grad v
            = (u, f)
    \end{align*}
    for almost all $t$ in $[0, T]$ and for all $v$ in $V_R$. A weak solution on $\R^2$ is defined for $u_0$ in $E_m$ with $u$ lying in $L^2([0, T]; \dot{H}^1)$ and $\prt_t u$ in $L^2([0, T]; V')$, and with ($NS$) holding for all $v$ in $V$.
\end{definition}

\Ignore{
For the Euler equations, existence is only known if the $L^p$-norm of the
initial vorticity is finite for some $p$ in $(1, \iny]$, and uniqueness is
known only under even stronger assumptions, such as the initial velocity lying
in $\Y$ (see also \cite{VIncompressible}). This is reflected in the following
definition of a weak solution to the Euler equations.
}

\begin{definition}[Weak Euler Solution]\label{D:WeakSolutionE}
    Given an initial velocity $u_0$ in $\Y(\Omega_R)$
    and forcing $f$ in $L^2_{loc}([0, \iny), H_R)$,
    $u$ in $L^\iny([0, T];
    H_R \cap H^1(\Omega_R))$ with $\prt_t u$ in $L^2([0, T]; V_R')$ is a weak solution to
    the Euler equations if $u(0) = u_0$ and
    \begin{align*}
        \mathbf{(E)}
        \qquad \int_{\Omega_R} \prt_t u \cdot v
            + \int_{\Omega_R} (u \cdot \grad u) \cdot v
            = (f, v)
    \end{align*}
    for almost all $t$ in $[0, T]$ and for all $v$ in $H_R \cap H^1(\Omega_R)$.  A weak solution on $\R^2$ is defined for $u_0$ in $\Y_m$ with $u$ lying in $L^\iny([0, T]; \Y_m)$  and $\prt_t u$ in $L^2([0, T]; V')$, and with ($E$) holding for all $v$ in $V$.
\end{definition}

Note that the test functions always have finite energy, even for solutions in the whole plane. (Test functions in $E_m$ would be too large to define the integrals involving the nonlinear terms in ($NS$) and ($E$).)

Given a solution to ($NS$), there exists a distribution $p$ (tempered, if the solution is in the whole plane) such that
\begin{align}\label{e:NSp}
    \prt_t u + u \cdot \grad u + \grad p = \nu \Delta u + f,
\end{align}
equality holding in the sense of distributions. This follows from a result of
Poincar\'{e} and de Rham that any distribution that is a curl-free vector is
the gradient of some scalar distribution.

Given a solution to ($E$), there exists a pressure $p$ such that
\begin{align}\label{e:Ep}
    \prt_t u + u \cdot \grad u + \grad p = f,
\end{align}
but we can only interpret $p$ as a distribution when working in the whole plane. Otherwise, we
must view $\prt_t u + u \cdot \grad u$ as lying in $H^{-1}(\Omega_R)$ and $p$
as lying in $L^2(\Omega_R)$. (\refE{Ep} follows, for instance, from Remark
I.1.9 p. 14 of \cite{T2001}.)

In both \refE{NSp} and \refE{Ep} the pressure is unique up to the addition of a
function of time. We resolve this ambiguity on $\Omega_R$ by requiring that
$\int_{\Omega_R} p(t) = 0$ and on $\R^2$ by requiring that $p(t)$ lie in
$L^2(\Omega_R)$ for almost all $t$ in $[0, T]$.

In referring to solutions on $\Omega_R$ we will say solutions for $R$ in $[1, \iny)$ and in referring to solutions on $\R^2$ we will say solutions for $R = \iny$.

\begin{theorem}\label{T:NSEProperties}
    \textbf{(1)} Assume that $u_0$ is in $E_m \cap \dot{H}^1$. There exists a unique weak solution $(u, p)$ to ($NS$) in the sense of
    \refD{WeakSolutionNS} with initial velocity $u_0$ for $R = \iny$ and initial
    velocity $\TR u_0$ for $R$ in $[1, \iny)$, with
    \begin{align}\label{e:NSNormBounds}
        \begin{array}{ll}
            u - \sigma_m \in L^\iny([0, T]; H_R),
                & \grad u \in L^\iny([0, T]; L^2(\Omega_R)).
%            t^{1/4} u \in L^2([0, T]; L^\iny(\Omega)),
%              & \sqrt{t} \Delta u \in L^2([0, T]; L^2(\Omega)), \\
%            \sqrt{t}(u - \sigma_m) \in L^2([0, T]; H^2(\Omega)),
%              & \sqrt{t} \prt_t u \in L^2([0, T]; H(\Omega)), \\
%            \sqrt{t} \grad p \in L^2([0, T]; L^2(\Omega)),
        \end{array}
    \end{align}
    Moreover, there is a bound on each of these norms that is independent of $R$ in $[1, \iny]$ and that
    depends continuously on $\norm{u_0}_\Z$.
    For $R = \iny$, if $u_0$ is in $E_m \cap \dot{H}^1$ and
    $\omega(f)$ is in $L^1([0, T]; L^{p_0} \cap L^\iny)$ then $\omega(u)$ is in $L^\iny([0, T]; L^{p_0} \cap L^\iny)$.
    
    \textbf{(2)} Assume that $u_0$ is in $\Y_m$ and that $\omega(f)$ is in $L^1([0, T]; L^{p_0} \cap L^\iny)$.
    There exists a unique weak solution $(u, p)$ to ($E$) in the sense of
    \refD{WeakSolutionE} with initial velocity $u_0$ for $R = \iny$ and initial
    velocity $\YR u_0$ for $R$ in $[1, \iny)$. We
    have,
    \begin{align*}
        \begin{array}{ll}
            u - \sigma_m \in L^\iny([0, T]; H_R),
                & \grad u \in L^\iny([0, T]; L^2(\Omega_R)), \\
            u \in L^\iny([0, T] \times \Omega_R),
                & u \in C([0, T] \times \ol{\Omega_R}), \\
            \prt_t u \in L^\iny([0, T]; H_R),
                & \grad p \in L^\iny([0, T]; L^2(\Omega_R)), \\
             \omega(u) \in L^\iny([0, T]; L^{p_0} \cap L^\iny(\Omega_R)),
        \end{array}
    \end{align*}
    and there is a bound on each of these norms that is independent of $R$ in $[1, \iny]$ and that
    depends continuously on $\norm{u_0}_\Y$.
    \Ignore{ % Ignore
    Also,
    \begin{align}\label{e:ConservationOfVorticityE}
        \norm{\omega(t)}_{L^q(\Omega_R)}
            \le \smallnorm{\omega^0}_{L^q(\Omega_R)}
    \end{align}
    for all $q$ in $[p_0, \iny]$ and almost all $t \ge 0$.
    If $R = \iny$, $p$ is in $L^\iny([0, T]; H^1(\R^2))$.
    } % End Ignore
\end{theorem}
\begin{proof}
These results are standard for $R < \iny$, except for the independence of the norms on $R$. The independence of the first norm in \refE{NSNormBounds} follows for solutions to ($NS$) from the energy inequality in \refE{NSEnergyBound} along with \refE{TRNSHighLimit}; for the second norm it follows from adapting slightly the proof of this same fact for finite energy in \cite{K2005ExpandingDomain}. The stronger bounds for solutions to ($E$) follow from the vorticity equation for ($E$). The results for $R = \iny$ are a minor modification of the same results for finite-energy: see, for instance, \cite{C1996}.
\end{proof}

\begin{definition}[Solution operators]\label{D:SolutionOperators}
Fix $f_R$ in $L^2([0, \iny); H_R)$ and write $f$ for $f_\iny$.
Let $S_R(t)$ be the solution operator for ($NS$) on $\Omega_R$ with $S = S_\iny$, and let $\ol{S}_R(t)$ be the solution operator for ($E$) on $\Omega_R$ with $\ol{S} = \ol{S}_\iny$.
\end{definition}

In application, we will often start with $f$ in $L^2([0, \iny); H)$ or even time-independent $f$ in $H$ and let $f_R = \PHR f$.

$S_R(0)$ is the identity operator, as is $\ol{S}_R(0)$. For all $t > 0$, $S_R$ maps $H_R \to V_R$ and $V_R \to V_R$ and $S$ maps $E_m \to E_m \cap \dot{H}^1$, $E_m \cap \dot{H}^1 \to E_m \cap \dot{H}^1$, $\X \to \Z$, and $\Z \to \Z$. For all $t \ge 0$, $\ol{S}_R$ maps $Y(\Omega_R) \to Y(\Omega_R)$ and $\ol{S}$ maps $\Y_m \to \Y_m$ and $\Y \to \Y$. Each of these maps is continuous.

Observe that $S_R(t) \TR u_0 = u(t)$ for $R$ in $[1, \iny)$ and $S(t) u_0 = u(t)$ for $R = \iny$ in part (1) of \refT{NSEProperties}, while $\ol{S}_R(t) \YR u_0 = u(t)$ for $R$ in $[1, \iny)$ and $\ol{S}(t) u_0 = u(t)$ for $R = \iny$ in part (2).

%
% Section
%
\section{Deterministic expanding domain limit}\label{S:Expanding}

\noindent First we establish the basic energy equality for deterministic solutions to ($NS$) in $\Omega_R$ and in all of $\R^2$. In all of $\R^2$, the energy is not finite, so we need to subtract $\sigma_m$ from the velocity to produce an ``energy'' equality. To make these estimates uniform over $R$ in $[1, \iny]$, we need, then, to subtract $\sigma_m$ from the velocity for $R < \iny$ as well. Actually, it will be slightly more convenient to subtract
\begin{align*}
	\ol{\sigma}_m =  \TR \sigma_m \text{ for } R \in [1, \iny), \quad
	\ol{\sigma}_m = \sigma_m \text{ for } R = \iny
\end{align*}
instead, but because $\ol{\sigma}_m \to \sigma_m$ in the $H^1(\Omega_R)$-norm as $R \to \iny$ by \refE{TRNSHighLimit}, this amounts to the same thing.

\begin{theorem}\label{T:EnergyEstimates}
	Let $u$ be a solution to ($NS$) as in \refD{WeakSolutionNS} with initial
	velocity $u_0$ in $H_R$ for $R < \iny$ and $u_0$ in $E_m$ for $R = \iny$.
	Then
	\begin{align}\label{e:NSEnergyEquality}
		\begin{split}
		&\norm{(u - \ol{\sigma}_m)(t)}_{L^2(\Omega_R)}^2
				+ 2 \nu \int_0^t \norm{\grad (u - \ol{\sigma}_m)}_{L^2(\Omega_R)}^2  \\
			&\quad= \norm{u_0 - \ol{\sigma}_m}_{L^2(\Omega_R)}^2
				- 2 \int_0^t \int_{\Omega_R} ((u - \ol{\sigma}_m) \cdot \grad \ol{\sigma}_m)
						\cdot (u - \ol{\sigma}_m) \\
 			&\quad\qquad
				- 2 \nu \int_0^t  \int_{\Omega_R} \grad \ol{\sigma}_m \cdot \grad (u - \ol{\sigma}_m)
				+ 2 \int_0^t \int_{\Omega} f_R \cdot (u - \ol{\sigma}_m)
		\end{split}
	\end{align}
	and
	\begin{align}\label{e:NSEnergyBoundGen}
			\begin{split}
		&\norm{(u - \ol{\sigma}_m)(t)}_{L^2(\Omega_R)}^2
						+ \nu \int_0^t \norm{\grad (u - \ol{\sigma}_m)}_{L^2(\Omega_R)}^2  \\
 			&\qquad\le \pr{\smallnorm{u_0 - \ol{\sigma}_m}_{L^2(\Omega_R)}^2 + C m^2 \nu t
				+ \norm{f}_{L^1([0, t]; L^2)}^2}^{1/2} e^{(Cm + 1)t},
			\end{split}
	\end{align}
	where $C$ depends only on $\norm{\grad \sigma_1}_{L^2 \cap L^\iny}$.
	 \Ignore{ % Ignore---interesting, but I don't need for expanding domain limit.
	  Also,
	\begin{align}\label{e:NSutBound}
		&\norm{\prt_t u}_{L^2([0, T]; V'(\Omega_R))}
			\le C(T, m).
	\end{align}
	Here, $m \mapsto C(T, m)$ is bounded over any finite interval.
	   } % End Ignore
\end{theorem}
\begin{proof}
    Assume first that $R < \iny$. Using $u - \ol{\sigma}_m$, which is in $V_R$ for all $t > 0$,
    as a test function in ($NS$) gives
    \begin{align*}
        \int_{\Omega_R} \prt_t u  &\cdot (u - \ol{\sigma}_m)
            + \int_{\Omega_R} (u \cdot \grad u) \cdot (u - \ol{\sigma}_m)
            + \nu \int_{\Omega_R} \grad u \cdot \grad (u - \ol{\sigma}_m) \\
            &= \int_{\Omega_R} f \cdot (u - \ol{\sigma}_m).
    \end{align*}

    But,
    \begin{align*}
        \int_{\Omega_R} \prt_t u \cdot (u - \ol{\sigma}_m)
            = \int_{\Omega_R} \prt_t (u - \ol{\sigma}_m) &\cdot (u - \ol{\sigma}_m)
            = \frac{1}{2} \frac{d}{dt} \norm{u -
                \ol{\sigma}_m}_{L^2(\Omega_R)}^2
    \end{align*}
    and
    \begin{align*}
        \int_{\Omega_R} &(u \cdot \grad u) \cdot (u - \ol{\sigma}_m)
            = \int_{\Omega_R}(u \cdot \grad (u - \ol{\sigma}_m)) \cdot (u - \ol{\sigma}_m) \\
            &\qquad\qquad
                + \int_{\Omega_R}(u \cdot \grad \ol{\sigma}_m) \cdot (u - \ol{\sigma}_m) \\
           &= \frac{1}{2} \int_{\Omega_R} u \cdot \grad \abs{u - \ol{\sigma}_m}^2
                + \int_{\Omega_R} ((u - \ol{\sigma}_m) \cdot \grad \ol{\sigma}_m) \cdot (u - \ol{\sigma}_m) \\
            &\qquad + \int_{\Omega_R} (\ol{\sigma}_m \cdot \grad \ol{\sigma}_m) \cdot (u -
                \ol{\sigma}_m).
    \end{align*}
    The first integral in the right-hand side above is formally zero because $\dv u
    = 0$ and $u \cdot \mathbf{n} = 0$ on $\prt \Omega_R$. More properly, we first observe
    that the integral is finite. This is because at time $t$ in $[0, T]$ both $u$
    and $u - \ol{\sigma}_m$ are in $H^1(\Omega_R)$ and so in $L^4(\Omega_R)$ by Sobolev
    embedding. But $\smallabs{\grad \abs{u - \ol{\sigma}_m}^2} \le 2\abs{u - \ol{\sigma}_m}
    \abs{\grad (u - \ol{\sigma}_m)}$, and applying \Holders inequality gives the
    finiteness of the integral. Approximating by smooth functions and using the
    dominated convergence theorem shows that the integral is zero. (This is the
    approach of Lemmas II.1.1 and II.1.3 p. 108-109 of \cite{T2001}.)
    Since by \refL{PVRsigmam}, $\ol{\sigma}_m \cdot \grad \ol{\sigma}_m$ is a gradient,
    the last integral above vanishes. We conclude that
    \begin{align*}
        \int_{\Omega_R} (u \cdot \grad u) \cdot (u - \ol{\sigma}_m)
           &= \int_{\Omega_R} ((u - \ol{\sigma}_m) \cdot \grad \ol{\sigma}_m) \cdot (u -
           \ol{\sigma}_m).
    \end{align*}

    For the final term, we observe that
    \begin{align*}
        \int_{\Omega_R} &\grad u \cdot \grad (u - \ol{\sigma}_m) \\
           &= \int_{\Omega_R} \grad (u - \ol{\sigma}_m) \cdot \grad (u - \ol{\sigma}_m)
              + \int_{\Omega_R} \grad \ol{\sigma}_m \cdot \grad (u - \ol{\sigma}_m).
    \end{align*}

    Combining all these equalities gives
    \begin{align*}
        \frac{1}{2} &\frac{d}{dt} \norm{u - \ol{\sigma}_m}_{L^2(\Omega_R)}^2
            + \nu \norm{\grad (u - \ol{\sigma}_m)}_{L^2(\Omega_R)}^2
                        \\
            &= - \int_{\Omega_R} ((u - \ol{\sigma}_m) \cdot \grad \ol{\sigma}_m) \cdot (u - \ol{\sigma}_m)
                 - \nu \int_{\Omega_R} \grad \ol{\sigma}_m \cdot \grad (u - \ol{\sigma}_m) \\
            &\qquad\qquad
                 + \int_{\Omega_R} f_R \cdot (u - \ol{\sigma}_m).
    \end{align*}
    Integrating in time gives \refE{NSEnergyEquality}. Also, we can bound the right-hand side by
     \begin{align*}
            &\norm{\grad \ol{\sigma}_m}_{L^\iny(\R^2)} \norm{u - \ol{\sigma}_m}_{L^2(\Omega_R)}^2
                        + \nu \norm{\grad \ol{\sigma}_m}_{L^2(\R^2)}
                         \norm{\grad (u - \ol{\sigma}_m)}_{L^2(\Omega_R)} \\
                    &\qquad+ \norm{f_R}_{L^2(\Omega_R)} \norm{u - \ol{\sigma}_m}_{L^2(\Omega_R} \\
            &\le 
                    C m \norm{u - \ol{\sigma}_m}_{L^2(\Omega_R)}^2
                    + C m^2 \nu
                    + \frac{1}{2} \norm{f_R}_{L^2(\Omega_R)}^2
                 + \frac{\nu}{2}
                         \norm{\grad (u - \ol{\sigma}_m)}_{L^2(\Omega_R)}^2,
    \end{align*}
    where we used Young's inequality and $\ol{\sigma}_m = m \sigma_1$. Thus,
    \begin{align}\label{e:BVDiffInequality}
    	\begin{split}
        \frac{d}{dt} &\norm{u - \ol{\sigma}_m}_{L^2(\Omega_R)}^2
                + \nu \norm{\grad (u - \ol{\sigma}_m)}_{L^2(\Omega_R)}^2 \\
            &\qquad\le C m^2 \nu +  \norm{f_R}_{L^2(\Omega_R)}^2
                + C \norm{u - \ol{\sigma}_m}_{L^2(\Omega_R)}^2.
       \end{split}
    \end{align}
    Integrating in time and applying Gronwall's inequality gives \refE{NSEnergyBoundGen}.
    
    The energy argument above works equally as well when $R = \iny$ with the
    exception of the term $(1/2) \int_{\R^2} u \cdot \grad \abs{u - \ol{\sigma}_m}^2$,
    which must be handled slightly differently, because at time $t$ in $[0, T]$ we
    no longer have $u$ in $H^1(\Omega_R)$, only in $\dot{H}^1(\R^2)$. So we divide
    the integral in two, writing
    \begin{align*}
        \frac{1}{2} \int_{\R^2} u \cdot \grad \abs{u - \ol{\sigma}_m}^2
           &= \frac{1}{2} \int_{\R^2} (u - \ol{\sigma}_m) \cdot \grad \abs{u - \ol{\sigma}_m}^2 \\
           &\qquad
                + \frac{1}{2} \int_{\R^2} \ol{\sigma}_m \cdot \grad \abs{u -
                    \ol{\sigma}_m}^2.
    \end{align*}
    The first integral is finite and, in fact, zero, using the same reasoning as
    with the similar term for $R < \iny$. The vector field $\ol{\sigma}_m$ is in
    $L^\iny(\R^2)$ and $\grad \abs{u - \ol{\sigma}_m}^2$ is in $L^1(\R^2)$ since
    $\smallabs{\grad \abs{u - \ol{\sigma}_m}^2} \le 2\abs{u - \ol{\sigma}_m} \abs{\grad
    (u - \ol{\sigma}_m)}$; hence, the second integral is finite. We then have
    \begin{align*}
       &\abs{\frac{1}{2} \int_{\R^2} \ol{\sigma}_m \cdot \grad \abs{u -
                    \ol{\sigma}_m}^2}
            = \lim_{R \to \iny}
                \abs{\frac{1}{2} \int_{\Omega_R} \ol{\sigma}_m \cdot \grad \abs{u -
                    \ol{\sigma}_m}^2} \\
           &\qquad= \lim_{R \to \iny}
                \abs{\frac{1}{2} \int_{\prt \Omega_R} (\ol{\sigma}_m \cdot \mathbf{n})
                        \abs{u - \ol{\sigma}_m}^2}.
    \end{align*}
    This integral vanishes, since $\ol{\sigma}_m \cdot \mathbf{n} = 0$ on $\prt \Omega_R$.
    \Ignore{ % Ignore
    
    To establish \refE{NSutBound}, let $v$ in $V(\Omega_R)$ have norm 1. Then
    from \refE{NSp},
\begin{align*}
	(\prt_t u_R, v) &= (- u_R \cdot \grad u_R - \grad p_R + \nu \Delta u_R + f_R, v) \\
		&= - (\dv (u_R \otimes u_R), v) - \nu(\grad u_R, \grad v) + (f_R, v) \\
		&= (u_R \otimes u_R, \grad v) - \nu(\grad u_R, \grad v) + (f_R, v),
\end{align*}
where the pairings are in the duality between $V(\Omega_R)$ and its dual. Then,
\begin{align*}
	&\abs{(\prt_t u_R, v)} 
			\le \norm{u_R \otimes u_R}_{L^2(\Omega_R)}
				+ \nu \norm{\grad u_R}_{L^2(\Omega_R)}
				+ \norm{f_R}_{L^2(\Omega_R)} \\
		&\qquad
		\le \norm{u_R \otimes u_R}_{L^2(\Omega_R)}
			+ \nu \norm{\grad (u_R - \ol{\sigma}_m)}_{L^2(\Omega_R)}
				   + \nu \norm{\grad \ol{\sigma}_m}_{L^2(\Omega_R)}
			+ \norm{f}_H.
\end{align*}
Using
\begin{align*}
	u_R \otimes u_R
		= (u_R - \ol{\sigma}_m) &\otimes (u_R - \ol{\sigma}_m)
		+ (u_R - \ol{\sigma}_m) \otimes \ol{\sigma}_m \\
		&+ \ol{\sigma}_m \otimes (u - \ol{\sigma}_m)
		+ \ol{\sigma}_m \otimes \ol{\sigma}_m,
\end{align*}
we have
\begin{align*}
	&\norm{u_R \otimes u_R}_{L^2(\Omega_R)}
		\le \norm{u_R - \ol{\sigma}_m}_{L^4(\Omega_R)}^2 +
			\norm{u_R - \ol{\sigma}_m}_{L^4(\Omega_R)} \norm{\ol{\sigma}_m}_{L^4(\Omega_R)}\\
		&\qquad
		+ \norm{\ol{\sigma}_m}_{L^4(\Omega_R)} \norm{u_R - \ol{\sigma}_m}_{L^4(\Omega_R)}
		+ \norm{\ol{\sigma}_m}_{L^4(\Omega_R)}^2.
\end{align*}
Now, $\norm{\ol{\sigma}_m}_{L^4(\Omega_R)} \le \norm{\ol{\sigma}_m}_{L^4(\R^2)} = m \norm{\sigma_1}_{L^4(\R^2)} = Cm$, while applying Ladyzhenskaya's inequality gives
\begin{align*}
	\norm{u_R - \ol{\sigma}_m}_{L^4(\Omega_R)}
		\le 2^{1/4} \norm{u_R - \ol{\sigma}_m}_{L^2(\Omega_R)}^{1/2}
			\norm{\grad(u_R - \ol{\sigma}_m)}_{L^2(\Omega_R)}^{1/2}.
\end{align*}
Thus,
\begin{align*}
	&\norm{u_R \otimes u_R}_{L^2(\Omega_R)}
		\le C \norm{u_R - \ol{\sigma}_m}_{L^2(\Omega_R)}
			\norm{\grad(u_R - \ol{\sigma}_m)}_{L^2(\Omega_R)} \\
			&\qquad
				+ Cm \norm{u_R - \ol{\sigma}_m}_{L^2(\Omega_R)}^{1/2}
			\norm{\grad(u_R - \ol{\sigma}_m)}_{L^2(\Omega_R)}^{1/2} 
				+ C m^2.
\end{align*}

Putting all of these bounds together, using \refE{NSEnergyBoundGen}, and applying Young's inequality gives \refE{NSutBound}.
} % End Ignore
\end{proof}

\begin{cor}\label{C:CEnergyBound}
	Let $f$ lie in $L^2([0, \iny); H)$ and let $f_R = \PHR f$.
	Let $u_0$ be in $E_m \cap \Z$ and
	let $u$ be a solution to ($NS$) as in \refD{WeakSolutionNS} with initial
	velocity $\TR u_0$ when $R$ is in $[1, \iny)$ and initial velocity $u_0$ when $R = \iny$.
	Then for a constant $C$ independent of $R$ and $u_0$,
	\begin{align}\label{e:NSEnergyBound}
			\begin{split}
		&\norm{(u - \ol{\sigma}_m)(t)}_{L^2(\Omega_R)}^2
						+ \nu \int_0^t \norm{\grad (u - \ol{\sigma}_m)}_{L^2(\Omega_R)}^2  \\
 			&\quad\le \pr{\smallnorm{u_0 - \sigma_m}_{V_R}^2 + C m^2 \nu t
				+ \norm{f}_{L^1([0, t]; L^2)}^2}^{1/2} e^{(Cm + 1)t}.
			\end{split}
	\end{align}
\end{cor}
\begin{proof}
Apply \refT{EnergyEstimates} with initial velocity $\TR u_0 = \TR (u_0 - \sigma_m) + \TR \sigma_m$ and use \refE{TRNSHighInequality}.
\end{proof}

\begin{remark}\label{R:TRVersusPR2}
Were we to use an initial velocity of $\UR u_0$ instead of $\TR u_0$ in \refC{CEnergyBound} we could replace $\smallnorm{u_0 - \sigma_m}_{V_R}$ on the right-hand side of \refE{NSEnergyBound} with $C \smallnorm{u_0 - \sigma_m}_{H_R}$. See \refR{TRVersusPR1}.
\end{remark}

We can control the decay of the tail of solutions to ($NS$) at time $t$ based, ultimately, on the their decay at time zero:
\begin{lemma}\label{L:Tail}
	For all $u_0$ in $\Z$,
	\begin{align*}
		\norm{S(t) u_0 - \sigma(S(t) u_0)}_{L^\iny([0, T]; L^2(\Omega_R^C))}
			\to 0 \text{ as } R \to \iny
	\end{align*}
	and	
	\begin{align*}
		\norm{S(t) u_0 - \sigma(S(t) u_0)}_{L^2([0, T]; H^1(\Omega_R^C))}
			\to 0 \text{ as } R \to \iny.
	\end{align*}
\end{lemma}
\begin{proof}
	This is a minor adaptation of Lemma 7.1 of \cite{K2005ExpandingDomain} to account for infinite-energy,
	and follows by a standard argument.
\end{proof}

\begin{theorem}\label{T:ExpandingDomainLimit}
Assume that $u_0$ lies in $\Z$ and let $u_R(t) = S_R(t) \TR u_0$ and $u(t) = S(t) u_0$. Then
\begin{align}\label{e:ExpDomVel}
	\norm{u_R - u}_{L^\iny([0, T]; H^1(\Omega_R))}
		\to 0 \text{ as } R \to \iny,
\end{align}
\begin{align}\label{e:ExpDomGradVel}
		\norm{\grad(u_R - u)}_{L^2([0, T]; L^2(\Omega_R))}
		\to 0 \text{ as } R \to \iny,
\end{align}
and
\begin{align}\label{e:FDiffBound}
	\norm{F(t, u) - F_R(t, u_R)}_{L^2([0, T]; V_R'(\Omega_R))}
		\to 0 \text{ as } R \to \iny.
\end{align}
In addition, the supremum over all $u_0$ in any bounded subset of $\Z$ and over all $R$ in $[1, \iny]$ of each of the quantities,
\begin{align}\label{e:FBound}
	\begin{split}
	&\norm{u_R - \sigma(u)}_{L^\iny([0, T]; L^2(\Omega_R))}, \,
	\norm{\grad u_R}_{L^2([0, T]; L^2(\Omega_R))}, \\
	&\norm{F_R(t, u_R)}_{L^2([0, T]; V_R'(\Omega_R))}
	\end{split}
\end{align}
is finite.
\end{theorem}
\begin{proof}
The first two bounds in \refE{FBound} follow from \refE{NSEnergyBound}. \refEAnd{ExpDomVel}{ExpDomGradVel} follow from Theorem 8.1 of \cite{K2005ExpandingDomain} extended to infinite energy solutions using \refE{NSEnergyBound}
% and \refE{NSutBound}
in place of the standard finite-energy energy bounds.

We now prove \refE{FDiffBound}. We have,
\begin{align*}
	&\norm{F(t, u) - F_R(t, u_R)}_{V_R'(\Omega_R)} 
				\le \norm{Au(t) - A_R u_R(t)}_{V_R'(\Omega_R)} \\
		&\qquad\qquad\qquad
				+ \norm{Bu(t) - B_R u_R(t)}_{V_R'(\Omega_R)}
				+ \norm{f- f_R}_{V_R'(\Omega_R)}.
\end{align*}
Let $v$ be in $V_R$ with $\norm{v}_{V_R} = 1$. Then
\begin{align*}
	&\abs{(Au(t) - A_R u_R(t), v)}
		= \nu \abs{(\Delta u(t) - \Delta u_R(t), v)} \\
		&\qquad= \nu \abs{(\grad u(t) - \grad u_R(t), \grad v)}
		\le \nu \norm{\grad u(t) - \grad u_R(t)}_{L^2(\Omega_R)}
			\norm{v}_{V_R}.
\end{align*}
Thus,
\begin{align*}
	\norm{Au(t) - A_R(u_R)(t)}_{L^2([0, T]; V_R'(\Omega_R))}
		\le \nu \norm{\grad u - \grad u_R}_{L^2([0, T]; L^2(\Omega_R))},
\end{align*}
which vanishes as $R \to \iny$ by \refE{ExpDomGradVel}.

For the nonlinear term,
\begin{align*}
	&\abs{(Bu(t) - B_R u_R(t), v)}
		= \abs{(u(t) \cdot \grad u(t) - u_R(t) \cdot \grad u_R(t), v)} \\
		&\qquad= \abs{(\dv (u(t) \otimes u(t) - u_R(t) \otimes u_R(t)), v)} \\
		&\qquad= \abs{(u(t) \otimes u(t) - u_R(t) \otimes u_R(t), \grad v)} \\
		&\qquad\le \norm{u(t) \otimes u(t) - u_R(t) \otimes u_R(t)}_{L^2(\Omega_R)}
			\norm{v}_{V_R}.
\end{align*}
But,
\begin{align*}
	&\norm{u(t)^i u(t)^j - u_R(t)^i \otimes u_R(t)^j}_{L^2(\Omega_R)} \\
		&\qquad \le \norm{u(t)^i (u(t)^j - u_R(t)^j)}_{L^2(\Omega_R)}
			+ \norm{(u(t)^i - u_R(t)^i) u_R(t)^j}_{L^2(\Omega_R)} \\
		&\qquad \le C \norm{u(t)}_{L^4(\Omega_R)}
			\norm{u(t) - u_R(t)}_{L^2(\Omega_R)} \\
		&\qquad \le C \norm{u(t)}_{L^2(\Omega_R)}^{1/2}  \norm{\grad u(t)}_{L^2(\Omega_R)}^{1/2}
			\norm{u(t) - u_R(t)}_{L^2(\Omega_R)} ^{1/2} \\
		&\qquad\qquad\qquad \times
				\norm{\grad(u(t) - u_R(t))}_{L^2(\Omega_R)} ^{1/2} \\
		&\qquad \le C  \norm{\grad u(t)}_{L^2(\Omega_R)}^{1/2}
				\norm{\grad (u(t) - u_R(t))}_{L^2(\Omega_R)} ^{1/2},
\end{align*}
where we used Ladyzhenskaya's inequality (which gives no dependence on $R$ for the constant $C$). Thus,
\begin{align*}
	&\norm{Bu(t) - B_R(u_R)(t)}_{L^2([0, T]; V_R'(\Omega_R))} \\
		&\qquad\le C \norm{\grad u}_{L^1([0, T]; L^2(\Omega_R)}
			\norm{\grad (u(t) - u_R(t))}_{L^1([0, T]; L^2(\Omega_R)} \\
		&\qquad C t \norm{\grad u}_{L^2([0, T]; L^2(\Omega_R)}
			\norm{\grad (u(t) - u_R(t))}_{L^2([0, T]; L^2(\Omega_R)},
\end{align*}
which vanishes as $R \to \iny$ by \refE{ExpDomGradVel}.

For the forcing term,
\begin{align*}
	\norm{f- f_R}_{L^2([0, T]; V_R'(\Omega_R))}
		\le \norm{f- f_R}_{L^2([0, T]; H_R(\Omega_R))},
\end{align*}
which vanishes as $R \to \iny$ by \refE{PHRConvergence}.

From these bounds, \refE{FDiffBound} follows.

It remains to establish the last bound in \refE{FBound}. But this follows from an argument similar to that we just made to prove \refE{FDiffBound}, using the first two bounds in \refE{FBound}.
\end{proof}

\begin{remark}\label{R:TRVersusPR3}
Together, the limits in \refEAnd{ExpDomVel}{ExpDomGradVel} are called the \textit{expanding domain limit} in \cite{K2005ExpandingDomain}. Most of the estimates involved in establishing these limits require only that the initial velocity lie in $H$ (or $\X$ for the infinite-energy extension). The key exception is that the regularity of the pressure is insufficient to complete the argument unless the initial velocity is in $V$ (or $\X^1$ for the infinite-energy extension).

Had we used $\UR$ in place of $\TR$ in defining the initial velocity, the expanding domain limit would still hold (indeed, this is how the limit was established in \cite{K2005ExpandingDomain}). An advantage of using $\UR$ is that the resulting bound on the rate of convergence is slightly improved, since $H^1$-norms are replaced by $L^2$-norms  in certain constants that appear in the bound. But this is unimportant in our use of the limit, so we preferred to use $\TR$, since it has a more natural definition.
\end{remark}

\begin{cor}\label{C:CorExpandingDomainLimit}
For all $u_0$ in $\Z$,
\begin{align}\label{e:ExpHR}
	\norm{S_R(t) \TR u_0 - \TR S(t) u_0}_{L^\iny([0, T]; H_R)}
		\to 0 \text{ as } R \to \iny
\end{align}
and
\begin{align}\label{e:ExpVR}
	\norm{S_R(t) \TR u_0 - \TR S(t) u_0}_{L^2([0, T]; V_R)}
		\to 0 \text{ as } R \to \iny.
\end{align}
\end{cor}
\begin{proof}
By \refEAnd{ExpDomVel}{ExpDomGradVel} 
\begin{align*}
	\lim_{R \to \iny} \norm{S_R(t) \TR u_0 - S(t) u_0}_{L^2([0, T]; V_R)}
		\to 0 \text{ as } R \to \iny.
\end{align*}
But by \refE{TRNSHighLimit},
\begin{align*}
	&\norm{\TR S(t) u_0 - S(t) u_0}_{L^2([0, T]; V_R)} \\
		&\qquad
		\le C \norm{S(t) u_0 - \sigma(S(t) u_0)}_{L^2([0, T]; H^1(\Omega_R \setminus \Omega_{R/2}))} +
			T^{1/2} \abs{m(u)} \beta(R) \\
		&\qquad
		\le C \norm{S(t) u_0 - \sigma(S(t) u_0)}_{L^2([0, T]; H^1(\Omega_{R/2}^C))} +
			T^{1/2} \abs{m(u)} \beta(R),
\end{align*}
which also vanishes as $R \to \iny$ by \refL{Tail}.
\refE{ExpVR} then follows from the triangle inequality. The proof of \refE{ExpHR} is similar.
\end{proof}

\begin{remark}
	It is only in the proof of \refC{CorExpandingDomainLimit} where we directly use the uniform decay
	over time of the tail of the velocity for solutions to ($NS$). It was, however, already used in the extension of
	the expanding domain limit from finite to infinite energy energy alluded to in the proof of
	\refT{ExpandingDomainLimit}.
\end{remark}

\begin{lemma}\label{L:FBoundNoTime}
For all $u$ in $\Z$.
\begin{align*}
	\norm{F(t, u) - F_R(t, \TR u)}_{V_R'(\Omega_R)}
		\to 0 \text{ as } R \to \iny,
\end{align*}
and the supremum over any bounded subset of $\Z$ and over all $R$ in $[1, \iny]$ of
\begin{align}\label{e:FBoundSingle}
	\norm{F_R(t, \TR u)}_{V_R'(\Omega_R)}
\end{align}
is finite.
\end{lemma}
\begin{proof}
The proof is the same as that of \refE{FDiffBound}, with no need to introduce the $L^2$-norm over $[0, T]$, and using \refL{ProjVR} in place of the bounds in \refEAnd{ExpDomVel}{ExpDomGradVel}.
\end{proof}

%
% Section
%
\Ignore{ % Ignore this section, but it is interesting.
\section{Where the energy goes as $R \to \iny$}\label{S:EnergyGoes}

\noindent \textbf{Might not need this section.}\medskip

\noindent Taking the vorticity of ($E$) for $u_0$ in $\Y_m$ gives  $\prt_t \ol{\omega} + \ol{u} \cdot \grad \ol{\omega} = 0$, and a simple integration by parts of $\prt_t \ol{\omega} + \ol{u} \cdot \grad \ol{\omega} = 0$ shows that
\begin{align*}
	\int_{\R^2} \ol{\omega}(t) &= m
		\text{ for all } t \ge 0.
\end{align*}
Similarly, for a solution to ($E$) with $u_0$ in $\Y(\Omega_R)$,
\begin{align}\label{e:mConservedE}
	 \int_{\Omega_R} \ol{\omega}(t)
	 	= \int_{\Omega_R} \ol{\omega}_0
		\text{ for all } t \ge 0.
\end{align}
In both cases, the total mass of the vorticity vanishes.
(These calculations are formal, but can be justified by a rather involved approximation argument.)

In contrast, for $t > 0$ all solutions to ($NS$) on $\Omega_R$ belong to $V_R$, so
\begin{align*}
	\int_{\Omega_R} \omega(t)
		= - \int_{\Omega_R} \dv u^\perp(t)
		= - \int_{\prt \Omega_R} u^\perp(t) \cdot \mathbf{n}
		= 0. 
\end{align*}
This integral vanishes regardless of the total mass of $\omega_0$.

Now, our plan is to start with $u_0$ in $E_m$, truncate $u_0$ to give $\TR u_0$ in $H_R$, and use $\TR u_0$ as the initial velocity for solutions to ($NS$) in $\Omega_R$. Similarly, for solutions to ($E$) we will start with $u_0$ in $\Y_m$ and $\TR u_0$ will be in $H_R \cap H^1(\Omega_R)$. In both cases, because of the first property of $\TR$ in \refL{Trunc}, the initial total mass of the vorticity will be $m$. For solutions to ($E$), this value is conserved (because there is no forcing) and we can think of the total mass of the vorticity as retaining a measure of how infinite the energy of the initial velocity in the whole space is. For solutions to ($NS$), however, the information is immediately lost after time zero. It turns out, however, that it remains encoded in another quantity, as we show in \refT{ApproximateConservationLaw}. This will be the critical fact used in the definition of the operator $\UR$ in \refD{UR}.

\begin{theorem}\label{T:ApproximateConservationLaw}
Assume either that: 1) $u_0$ lies in $E_m$ and $u = u_R$ is the solution to $(NS)$ on $\Omega_R$ with initial velocity $\TR u_0$, or 2) $u_0$ lies in $\Y_m$ and $\ol{u}$ is the solution to ($E$) on $\R^2$ with initial velocity $u_0$. Write $\omega = \omega(u)$ and $\ol{\omega} = \omega(\ol{u})$. Then for all $\nu > 0$,
\begin{align}
	\lim_{R \to \iny} \frac{(u(t), \sigma_1)}{\log R}
		&= \frac{m}{2 \pi}, \,
	\lim_{R \to \iny} \frac{(\omega(t), \psi_{\sigma_1})}{\log R}
		= - \frac{m}{2 \pi},
			\label{e:LimitNS} \\
	\lim_{R \to \iny} \frac{(\ol{u}(t), \sigma_1)}{\log R}
		&= \frac{m}{2 \pi}, \,
	\lim_{R \to \iny} \frac{(\ol{\omega}(t), \psi_{\sigma_1})}{\log R}
		= 0, 
		\label{e:LimitE}
\end{align}
the limits each being uniform over all $t$ in $(0, T]$ with a rate of convergence of order $(\log R)^{-1/2}$.\end{theorem}
\begin{proof}
By \refT{NSEProperties},
\begin{align*}
	f_{R, t}(m)
		:= \smallnorm{u (t)- \sigma_m}_{L^2(\Omega_R)}^2
		= \smallnorm{u(t) - m \sigma_1}_{L^2(\Omega_R)}^2
\end{align*}
is bounded uniformly over $(R, t)$ in $[1, \iny) \times (0, T]$.

\Ignore{ % Ignore
We know from Theorem 1.2 that the value of $m$ comes from the $E_m$ space in which $u_0$ lies. But another way of calculating it is to determine, for any given value of $R$ and $t$, the value $m = m_{R, t}$ that minimizes $f_{R, t}$ and look for the limit as $R \to \iny$ of $m_{R, t}$.
} % End Ignore

Now, $f_{R, t}$ is a perfectly well-defined differentiable function of one variable, having no maximum (since $f_{R, t}$ diverges to infinity as $m \to \pm \iny$), so we can find its minimum using simple calculus. We have
\begin{align*}
	f_{R, t}'(m)
		= 2(u(t)  - m \sigma_1, - \sigma_1).
\end{align*}
Setting $f_{R, t}' = 0$ gives a minimum at
\begin{align*}
	m_{R, t} = \frac{(u (t), \sigma_1)}{\norm{\sigma_1}_{L^2(\Omega_R)}^2}.
\end{align*}
Thus, $f_{R, t}(m_{R, t}) \le f_{R, t}(m)$ so $f_{R, t}(m_{R, t})$, like $f_{R, t}(m)$, is bounded uniformly over $(R, t)$ in $[1, \iny) \times (0, T]$. But by the triangle inequality,
\begin{align*}
	\norm{\sigma_{m_{R, t}} - \sigma_m}_{L^2(\Omega_R)}
		\le f_{R, t}(m_{R, t})^{1/2} + f_{R, t}(m)^{1/2},
\end{align*}
meaning that $\smallnorm{\sigma_{m_{R, t}} - \sigma_m}_{L^2(\Omega_R)}$ is bounded uniformly over $(R, t)$ in $[1, \iny) \times (0, T]$. From \refE{sigma1Norm},
\begin{align*}
	&\norm{\sigma_{m_{R, t}} - \sigma_m}_{L^2(\Omega_R)}
		= \abs{m_{R, t} - m} \norm{\sigma_1}_{L^2(\Omega_R)} \\
		&\qquad
		\le \abs{m_{R, t} - m} \pr{C_3 + \frac{\sqrt{2}}{2 \pi}( \log R)^{1/2}}.
\end{align*}
It follows that
\begin{align*}
	\abs{m_{R, t} - m}
		\le C (\log R)^{-1/2},
\end{align*}
where $C$ depends on $\nu$, $T$, and $u_0$.

This gives the first limit in \refE{LimitNS}. For the second limit, we observe that
\begin{align}\label{e:usop}
	\begin{split}
	(u, \sigma_1)
		&= (u, \grad^\perp \psi_{\sigma_1})
		= - ((u)^\perp, \grad \psi_{\sigma_1}) \\
		&= (\dv (u)^\perp, \psi_{\sigma_1})
			- \int_{\prt \Omega_R} ((u) ^\perp \cdot \mathbf{n}) \psi_{\sigma_1} \\
		&= - (\omega^{\nu, R}, \psi_{\sigma_1}).
	\end{split}
\end{align}
Here we used the vanishing of $u^{\nu, R}$ on $\prt \Omega_R$ for all $t > 0$.

The first limit in \refE{LimitE} follows using similar reasoning as for the first limit in \refE{LimitNS}. For the second limit, we perform the same calculation in \refE{usop} with $\ol{u}$ in place of $u^{\nu, R}$, using that $\psi_{\sigma_1}$ is constant on $\prt \Omega_R$. This gives
\begin{align*}
	(\ol{u}, \sigma_1)
		&= - (\ol{\omega}, \psi_{\sigma_1}) - \psi_{\sigma_1}(R) \int_{\Omega_R} \dv \ol{u}^\perp
		= - (\ol{\omega}, \psi_{\sigma_1}) + \psi_{\sigma_1}(R) \int_{\Omega_R} \ol{\omega}.
\end{align*}
Then from \refE{mConservedE} and the first property of $\TR$ in \refL{Trunc},
\begin{align*}
	(\ol{\omega}, \psi_{\sigma_1}) 
		= -(\ol{u}, \sigma_1) + m \psi_{\sigma_1}(R)
\end{align*}
and thus using \refE{psiForrLarge} that
\begin{align*}
	\frac{(\ol{\omega}, \psi_{\sigma_1})}{\log R} 
		= -\frac{(\ol{u}, \sigma_1)}{\log R}
				+ m \frac{C_2 +  \frac{1}{2 \pi} \log R}{\log R}.
\end{align*}
Since the limit of the left-hand side is $m/(2 \pi)$, the second limit in \refE{LimitE} follows.
\end{proof}

\begin{definition}[Approximate inverse to $\TR$]\label{D:UR}
When applied to solutions to ($NS$), we define the operator $\UR \colon V_R \to \X \cap \dot{H}^1(\R^2)$ by
\begin{align*}
	\UR(u) = \set{v  \in E_m  \cap \dot{H}^1(\R^2) \colon v|_{\Omega_R} = v},
\end{align*}
where
\begin{align*}
	m
		= 2 \pi \frac{(v, \sigma_1)}{\log R}
		= - 2 \pi \frac{(\omega(v), \psi_{\sigma_1})}{\log R}.
\end{align*}
When applied to solutions to ($E$), we define the operator $\UR \colon \Y \to \Y(\Omega_R)$ by
\begin{align*}
	\UR(u) = \set{v  \in Y_m \colon v|_{\Omega_R} = v},
\end{align*}
where
\begin{align*}
	 m = \int_{\Omega_R} \omega(v).
\end{align*}
\end{definition}
} % End of Ignore of entire section

%
% Section
%
\section{Definition of infinite-energy statistical solutions}\label{S:SSNS}

\noindent Following \cite{FMRT} p. 264-265, we define a statistical solution to ($NS$) on $\Omega_R$, first defining the space of test functions.

\begin{definition}\label{D:Test}
The space $\Test_R$ of test functions, $R < \iny$, is the set of all functions $\Phi \colon H_R \to \R$ such that $\Phi(u) = \phi((u, g_1), \dots, (u, g_k))$ for some $\phi$ in $C^1(\R^k)$ and some $g_1, \dots, g_k$ in $V_R$. The  \Frechet derivative of such a $\Phi$ is given by
\begin{align*}
	\Phi'(u) = \sum_{j = 1}^k \prt_j \phi((u, g_1), \dots, (u, g_k)) g_j,
\end{align*}
which lies in $V_R$ since each $g_j$ is in $V_R$. When $R = \iny$, we require that each of $g_1, \dots, g_k$ be compactly supported in $\R^2$, so that $\Phi \colon \X \to \R$, and we also write $\Test$ for $\Test_\iny$. This is the same class of test functions as for homogeneous solutions in the whole space (Definition 2.3 p. 278 of \cite{FMRT}).
\end{definition}

Observe that because each $\prt_j \phi$ is bounded,
\begin{align}\label{e:PhipBound}
	\norm{\Phi'(u)}_{V_R} \le C(\Phi)
\end{align}
for all $u$ in $V_R$.

For $t \ge 0$ and $u$ in $H_R$ let
\begin{align*}
	F_R(t, u)
		= f_R(t) - \nu A_R u - B_R(u),
\end{align*}
where $A_R$ is the Stokes operator and $B_R$ is the classical linear operator associated with the nonlinear term in ($NS$) on $\Omega_R$. (See, for instance, p. 38 of \cite{FMRT}.) We also write $F$ for $F_\iny$, $A$ for $A_\iny$, and $B$ for $B_\iny$.

For $R = \iny$, we will assume for simplicity that $f$ is in $L^2_{loc}([0, \iny); H)$; that is, we do not allow infinite forcing.

\begin{definition}[Statistical solution to ($NS$)]\label{D:SSNS}
Assume that $\mu_0$ is a Borel probability measure on $H_R$. Then a family,
\begin{align*}
	\mu = \set{\mu_t}_{t \ge 0},
\end{align*}
of Borel probability measures on $H_R$, $R < \iny$, is a statistical solution to ($NS$) (SSNS) on $\Omega_R$ if each of the following is satisfied:
\begin{enumerate}
\item For all $\Phi$ in $\Test_R$ and all $t \ge 0$,
\begin{align*}
\int_{H_R} \Phi(u) \, d \mu_t(u)
	= \int_{H_R} &\Phi(u) \, d \mu_0(u)  \\
		&+ \int_0^t \int_{H_R} (F_R(s, u), \Phi'(u)) \, d \mu_s(u) \, ds.
\end{align*}

\item
For all $t \ge 0$,
\begin{align*}
\int_{H_R} \norm{u}_{L^2}^2 \, &d \mu_t(u)
	+ 2 \nu \int_0^t \int_{H_R} \norm{\grad u}_{L^2}^2 \, d \mu_s(u) \, ds \\
		&= \int_0^t \int_{H_R} (f(s), u) \, d \mu_s(u) \, ds
			+ \int_{H_R} \norm{u}_{L^2}^2 \, d \mu_0(u).
\end{align*}

\item
The map
\begin{align*}
	t \mapsto \int_{H_R} \phi(u) \, d \mu_t(u)
\end{align*}
is measurable for all $t \ge 0$ and all $\phi$ in $C^0(H_R)$.

\item
The map
\begin{align*}
	t \mapsto \int_{H_R} \norm{u}_{H_R}^2 \, d \mu_t(u)
\end{align*}
lies in $L^\iny_{loc}([0, \iny))$.

\item
The map
\begin{align*}
	t \mapsto \int_{H_R} \norm{\grad u}_{L^2}^2 \, d \mu_t(u)
\end{align*}
lies in $L^1_{loc}([0, \iny))$.
\end{enumerate}

When $R = \iny$, we make two changes in the definition. First, we replace $H_R$ by $\X$ throughout. Second, the energy equality in (2) is replaced by
\begin{itemize}
\item[(2')]
For all $t \ge 0$,
\begin{align*}
\int_\X &\norm{u - \sigma(u)}_{L^2}^2 \, d \mu_t(u)
	+ 2 \nu \int_0^t \int_\X \norm{\grad (u - \sigma(u))}_{L^2}^2 \, d \mu_s(u) \, ds \\
		&= \int_0^t \int_\X (f(s), u - \sigma(u)) \, d \mu_s(u) \, ds
			+ \int_\X \norm{u - \sigma(u)}_{L^2}^2 \, d \mu_0(u) \\
			&\quad- 2 \int_0^t \int_\X ((u - \sigma(u)) \cdot \grad \sigma(u)), u - \sigma(u))
						\, d \mu_s(u) \, ds \\
 			&\quad\qquad
				- 2 \nu \int_0^t  \int_\X \grad \sigma(u) \cdot \grad (u - \sigma(u))
					\, d \mu_s(u) \, ds,
\end{align*}
\end{itemize}
where $\sigma(u)$ is defined in \refE{L2Part}.
\end{definition}

The following is from Theorems 1.1 and 1.2 Chapter V of \cite{FMRT}: 

%--- Theorem
\begin{theorem}\label{T:FMRT}
Let $\mu_0$ be as in \refD{SSNS} , $R < \iny$, with kinetic energy
\begin{align*}
	 \int_{H_R} \norm{u}_{H_R}^2 \, d \mu_0(u) < \iny
\end{align*}
and assume that $f$ lies in $L^2_{loc}([0, \iny); H_R)$. There exists a SSNS, $\mu$, as in \refD{SSNS}. If the support of $\mu_0$ is $(H_R, V_R)$-bounded as in \refD{BoundedSupport} (meaning that the containment in \refE{FMRT115} holds for $t = 0$) and $f$ in $H_R$ is time-independent then $\mu_t = S_R(t) \mu_0$ for all $t \ge 0$ is a SSNS. Furthermore, this solution is the unique SSNS satisfying \refEThrough{FMRT114}{FMRT116}:
\begin{align}
	&t \mapsto \int_{H_R} \varphi(u) \, d \mu_t(u)
		\text{ is continuous  on } [0, \iny)
		\text{ for all } \varphi \text{ in } C(H_R^w), \label{e:FMRT114} \\
	&\supp \mu_t \subseteq \set{u \in V_R \colon \norm{u}_{H_R} \le M}
		\text{ for all } t \ge 0 \text{ for some } M,\label{e:FMRT115} \\
	&\int_{H_R} \Psi(t, u) \, d \mu_t(u)
		= \int_{H_R} \Psi(0, u) \, d \mu_0(u) \nonumber \\
		&\qquad\qquad
				+ \int_0^t \int_{H_R}
					\brac{\Psi_s'(s, u) + (F(u), \Psi_u'(s, u))}
						\, d \mu_s(u) \, ds. \label{e:FMRT116}
\end{align}
$H_R^w$ is the space $H_R$ in the weak topology. In \refE{FMRT115}, equality holds for all Fr\'{e}chet-differentiable  continuous real-valued functions on $[0, \iny) \times V_R$ (see the discussion following Equation V.1.16 in \cite{FMRT} for more details).
\end{theorem}

For statistical solutions to ($E$), we consider only solutions in the whole plane. For solutions to ($E$) there is no term involving the Stokes operator, so we define
\begin{align*}
	F(t, u)
		= f(t) - B(u).
\end{align*}

\begin{definition}[Statistical solution to ($E$) in the plane]\label{D:SSE}

Assume that $\mu_0$ is a Borel probability measure on $\X$. A statistical solution to the Euler equations (SSE) on $\X$ satisfies all the properties of a SSNS in \refD{SSNS} for $R = \iny$ except that the terms involving $\nu$ in property (2') are eliminated.
\end{definition}

%
% Section
%
\section{Construction of Navier-Stokes solutions}\label{S:ConstructNS}

%nh\noindent It is possible to extend \refT{FMRT} to allow $R = \iny$ by adapting the proof of it on p. 313-321 of \cite{FMRT}. Adapting the uniqueness proof is simple, as it only requires using the energy bound in \refE{NSEnergyBound} in place of that used in \cite{FMRT}. Adapting the existence proof is more difficult, as certain embeddings that are compact for $R < \iny$ are no longer compact. In any case, our concern here is with higher regularity solutions appropriate for making vanishing viscosity arguments, so we can take another approach, which will lead us to \refT{FMRTInf}.

% \noindent We first make some definitions and observations to motivate the proof.

\noindent Let $S(t)$ be the solution operator on $\X$ as in \refD{SolutionOperators}. Given that we expect the analog of \refT{FMRT} to hold for infinite-energy solutions in $\R^2$, we would expect that
\begin{align}\label{e:mut}
	\mu_t = S(t) \mu_0
\end{align}
is the unique SSNS associated to the initial measure $\mu_0$ if we assume that the support of the initial Borel probability measure $\mu_0$ is $(\X, \Z)$-bounded as in \refD{BoundedSupport}. We show that this is, in fact, the case. Our approach will be to use the SSNS on $\Omega_R$ and take a limit as $R \to \iny$ in a careful way to demonstrate that $\mu_t$ is a SSNS on all of $\R^2$.

We start by defining the initial probability measure $\mu_0^R$ on $H_R$ by
\begin{align*}
	\mu_0^R (E) = \mu_0(\TRinv E)
\end{align*}
for all Borel measurable subsets $E$ of $H_R$. Then $\mu_0^R$ is a probability measure, for $\mu_0^R(H_R) = \mu_0(\TRinv H_R) = \mu_0(\X) = 1$. Since we are treating initial probability distributions supported on $\X^1$, we use projection into $V_R$. When working with SSNSs as weak as those of \refD{SSNS}, projection into $H_R$ would be used instead (though the limiting argument in that case is considerably more involved).

Similarly, we define the forcing term $f_R$ in $F_R$ by letting
\begin{align}\label{e:fR}
	f_R = \mathbf{P}_{H_R} f,
\end{align}
where $\mathbf{P}_{H_R}$ is projection into $H_R$.
For simplicity, we assume that $f$ is time-independent. Then
\begin{align}\label{e:fConvergence}
	\norm{f - f_R}_{H_R} \to 0
		\text{ as } R \to \iny
\end{align}
from Lemma 4.2 of \cite{K2005ExpandingDomain} and the observation that projection into $H_R$ gives the closest element in $H_R$.

We let $\mu^R$ be the associated SSNS on $\Omega_R$, so that, by \refT{FMRT},
\begin{align*}
	\mu_t^R = S_R(t) \mu_0^R,
\end{align*}
meaning that $S_R(t) \mu_0^R(E) = \mu_0^R(S_R^{-1}(t) E)$ for any Borel measurable subset $E$ of $H_R$.

Let $\Phi$ be in $\Cal{T}$ as in \refD{Test}. Because each $g_j$ is compactly supported, for all sufficiently large $R$, we can define a test function $\Phi_R$ in $\Cal{T}_R$ by
\begin{align}\label{e:PhiRPhi}
	\Phi_R(v)
		\EqDef \phi((v, g_1|_{\Omega_R}), \dots, (v, g_k|_{\Omega_R}))
		= \Phi(\Cal{E}_R v)
\end{align}
for all $v$ in $H_R$, where $\Cal{E}_R v$ is extension by zero of $v$ in $H_R$ to all of $\R^2$. It follows that for all $v$ in $H_R$,
\begin{align}\label{e:PhipRPhip}
	\Cal{E}_R \Phi_R'(v) = \Phi'(\Cal{E}_R v).
\end{align}

From now on, we always assume that $R$ is sufficiently large that \refE{PhiRPhi} holds.

For all $u$ in $\Z$,
\begin{align*}
	&\abs{\Phi_R(\TR u) - \Phi(u)} \\
		&\qquad
			= \abs{\phi((\TR u, g_1|_{\Omega_R}), \dots, (\TR u, g_k|_{\Omega_R}))
				- \phi((u, g_1), \dots, (u, g_k))} \\
		&\qquad
			\le \norm{\grad \phi}_{L^\iny} \abs{(\TR u - u, g_1), \dots, (\TR u - u, g_k)} \\
		&\qquad
			\le \norm{\phi}_{C^1}
			\norm{\TR u - u}_{L^2(\Omega_R)} \pr{\norm{g_1}_H^2 + \cdots + \norm{g_k}_H^2}^{1/2} \\
		&\qquad
			\le C \norm{\TR u - u}_{L^2(\Omega_R)}.
\end{align*}
Thus from \refL{ProjVR},
\begin{align}\label{e:PhiPhiRTR}
	\Phi_R(\TR u)
	 		\to \Phi(u) 
			\text{ as R } \to \iny.
\end{align}

Similarly, for all $u$ in $\Z$,
\begin{align*}
	&\norm{\Phi'_R(\TR u) - \Phi'(u)}_{V_R} \\
		&\quad
			\le \sum_{j=1}^k \abs{\prt_j \phi((\TR u, g_1), \dots, (\TR u, g_k))
					- \prt_j \phi((u, g_1), \dots, (u, g_k))} \norm{g_j}_{V_R} \\
		&\quad
			\le C \sum_{j=1}^k \mu_j \pr{\abs{(\TR u, g_1), \dots, (\TR u, g_k)
					- (u, g_1), \dots, (u, g_k)}} \\
		&\quad
			= C \sum_{j=1}^k \mu_j \pr{\abs{(\TR u - u, g_1), \dots, (\TR u - u, g_k)}} \\
		&\quad
			\le C \sum_{j=1}^k \mu_j \pr{\norm{\TR u - u}_{H_R} 
				 \pr{\norm{g_1}_H^2 + \cdots + \norm{g_k}_H^2}^{1/2}},
\end{align*}
where $\mu_j$ is the modulus of continuity of $\prt_j \phi$. Thus by \refL{ProjVR},
\begin{align}\label{e:PhipPhipR}
	\norm{\Phi_R'(\TR u) - \Phi'(u)}_{V_R} \to 0
		\text{ as } R \to \iny.
\end{align}

Before proceeding, we mention one logical simplification that we cannot make. It might seem reasonable to try to show that
\begin{align}\label{e:MightBeReasonable}
	\begin{split}
	S(t) \mu_0
		&= \lim_{R \to \iny}  \TR \circ S_R(t) \circ \TRinv \mu_0 \\
		&= \lim_{R \to \iny} \mu_0 \circ \TRinv \circ S_R(t)^{-1} \circ \TR
	\end{split}
\end{align}
by showing that equality holds on any Borel measurable set $E$. We note, however, that if $\mu_0$ is supported on a singleton set $E = \set{u_0}$ in $\X$, $u_0$ nonzero, then $\TRinv \circ S_R(t)^{-1} \circ \TR (S_R(t) u_0)$ will in general never equal $u_0$. Thus, the right-hand side above applied to $E$ will evaluate to 0 for all $R$, while the left-hand side will evaluate to 1.

Observe that \refE{MightBeReasonable} is equivalent to saying that
 \begin{align}\label{e:MightBeReasonableEquivalent}
 	\begin{split}
	\int_\X \Phi(u) \, &d \mu_t(u)
		=  \lim_{R \to \iny} \int_{H_R} \Phi \pr{u} \, d \nu_t^R(u),
	\end{split}
\end{align}
where $\nu_t = \TR \mu_t$ for all test functions $\Phi$ (which are dense in the set of bounded continuous functions). We will prove instead that
 \begin{align}\label{e:limPhimut}
 	\begin{split}
	\int_\X \Phi(u) \, &d \mu_t(u)
		=  \lim_{R \to \iny} \int_{H_R} \Phi_R \pr{u} \, d \mu_t^R(u),
	\end{split}
\end{align}
and that similar limits hold for the other integral in Property (1) of \refD{SSNS}, thus circumventing our difficulty. In this weak sense, the expanding domain limit could be said to hold for statistical solutions.

\refT{FMRT} continues to hold for $R = \iny$ if we impose at the outset the condition that the initial velocity $\mu_0$ is supported in $\Z$. This condition is required to allow us to take advantage of the expanding domain limit and related bounds from \refS{Expanding}. The idea for proving existence is to first assume that the measure has bounded support in $\Z$, apply the results of \refS{Expanding}, which require such support, then use the linearity of properties (1)-(5) of \refD{SSNS} to drop the boundedness assumption. This leads to \refT{FMRTInf}.

%--- Theorem
\begin{theorem}\label{T:FMRTInf}
Let $\mu_0$ be as in \refD{SSNS} for $R = \iny$, but supported in $\Z$, and having ``energy''
\begin{align*}
	 \int_\X \norm{u}_\X^2 \, d \mu_0(u) < \iny.
\end{align*}
Assume that $f$ lies in $L^2_{loc}([0, \iny); V)$. There exists a SSNS, $\mu$, as in \refD{SSNS}. If the support of $\mu_0$ is $(\X, \Z)$-bounded as in \refD{BoundedSupport}  (meaning that \refE{FMRT115Inf} holds for $t = 0$) and $f$ in $V$ is time-independent then $\mu_t = S(t) \mu_0$ for all $t \ge 0$ is a SSNS. Furthermore, this solution is the unique SSNS satisfying \refEThrough{FMRT114Inf}{FMRT116Inf}:
\begin{align}
	&t \mapsto \int_\X \varphi(u) \, d \mu_t(u)
		\text{ is continuous  on } [0, \iny)
		\text{ for all } \varphi \text{ in } C(H^w), \label{e:FMRT114Inf} \\
	&\supp \mu_t \subseteq \set{u \in \Z \colon \norm{u}_\X \le M(t)}
		\text{ for all } t \ge 0,\label{e:FMRT115Inf} \\
	&\int_\X \Psi(t, u) \, d \mu_t(u)
		= \int_\X \Psi(0, u) \, d \mu_0(u) \nonumber \\
		&\qquad\qquad
				+ \int_0^t \int_X
					\brac{\Psi_s'(s, u) + (F(u), \Psi_u'(s, u))}
						\, d \mu_s(u) \, ds. \label{e:FMRT116Inf}
\end{align}
In \refE{FMRT115Inf}, $M$ is continuous on $[0, \iny)$. In \refE{FMRT115Inf}, equality holds for all Fr\'{e}chet-differentiable  continuous real-valued functions on $[0, \iny) \times V$.
\end{theorem}
\begin{proof}
\textbf{Existence}:
Assume first that the initial Borel probability measure $\mu_0$ is bounded in $\X^1$, meaning that
\begin{align*}
	\supp \mu_0 \subseteq \set{u \in \X^1 \colon \norm{u}_{\X^1} \le M},
		\text{ for some } M,
\end{align*}
and define $\mu_t$ by \refE{mut} for $t \ge 0$. By \refE{NSNormBounds} it follows that
\begin{align}\label{e:mutBoundedE}
	\supp \mu_t \subseteq \set{u \in \X^1 \colon \norm{u}_{\X^1} \le M(t)}
\end{align}
for some continuous function $M$.
%which in turn means that the bounds in \refE{FBound} and \refE{FBoundSingle} are in effect uniform over all $R$ in $[1, \iny]$ and $(t, u_0)$ or $(t, u)$ in $[0, T] \times \X$.

In \refT{ExpandingDomainLimit}, for initial velocity $u_0$ in $\Z$ we defined $u_R(t) = S_R(t) \TR u_0$ and $u(t) = S(t) u_0$. In this proof we will be integrating over all initial velocities in $\X$ and calling the initial velocity $u$, to agree with the notation of \cite{FMRT}. In this notation, \refE{FBound} and \refE{FBoundSingle}  become
\begin{align}\label{e:FBoundAll}
	\begin{split}
		\begin{array}{ll}
	\norm{S_R(t) \TR u - \sigma(u)}_{L^\iny([0, T]; L^2(\Omega_R))}, &
		\norm{\grad S_R(t) \TR u}_{L^2([0, T]; L^2(\Omega_R))}, \\
	\norm{F_R(t, S_R(t) \TR u)}_{L^2([0, T]; V_R'(\Omega_R))}, &
			\norm{F_R(t, \TR S(t) u)}_{V_R'(\Omega_R)}
		\end{array}
	\end{split}
\end{align}
are bounded on $\supp \mu_0$ uniformly over all $R$ in $[1, \iny]$. This will allow us to apply the dominated convergence theorem in several steps in our proof. 

Using \refE{PhiRPhi}, for all $u$ in $\Z$,
\begin{align}\label{e:PhiPhiR}
	\begin{split}
	\Phi &\pr{\lim_{R \to \iny} S_R(t) \TR u}
		=  \Phi \pr{\lim_{R \to \iny} \Cal{E}_R S_R(t) \TR u} \\
		&\qquad=  \lim_{R \to \iny} \Phi \pr{\Cal{E}_R S_R(t) \TR u}
		=  \lim_{R \to \iny} \Phi_R \pr{S_R(t) \TR u},
	\end{split}
\end{align}
where we used the continuity of $\Phi$. Thus we have,
 \begin{align*}
 	\begin{split}
	\int_\X \Phi(u) \, &d \mu_t(u)
		= \int_\X \Phi(S(t) u) \, d \mu_0(u) \\
		&= \int_\X \Phi \pr{\lim_{R \to \iny} S_R(t) \TR u} \, d \mu_0(u) \\
		&=  \int_\X \lim_{R \to \iny}  \Phi_R \pr{S_R(t) \TR u} \, d \mu_0(u)  \\
		&=  \lim_{R \to \iny} \int_\X \Phi_R \pr{S_R(t) \TR u} \, d \mu_0(u)  \\
		&=  \lim_{R \to \iny} \int_{H_R} \Phi_R \pr{S_R(t) v} \, d \mu_0^R(v)  \\
		&=  \lim_{R \to \iny} \int_{H_R} \Phi_R \pr{u} \, d \mu_t^R(u)
	\end{split}
\end{align*}
giving \refE{limPhimut}.
The first equality follows from \refE{mut}, since the space of bounded continuous functions is dual to the space of Borel probability measures. The limit in the second equality follows from \refT{ExpandingDomainLimit}. The third equality follows from \refE{PhiPhiR}. The fourth equality follows by the dominated convergence theorem, since $\Phi_R$ is uniformly bounded over $R$ in $[1, \iny]$ and $\mu_0$ is a finite measure. The fifth equality follows from \refL{IntXToHR}. The sixth and final equality follows in the same way as does the first.

This shows that \refE{limPhimut} holds for all $t \ge 0$, so if we can show that
\begin{align}\label{e:FToShow}
	\begin{split}
	\int_0^t \int_\X &(F(s, u), \Phi'(u)) \, d \mu_s(u) \, ds \\
		&= \lim_{R \to \iny} \int_0^t \int_{H_R} (F_R(s, u), \Phi_R'(u)) \, d \mu^R_s(u) \, ds
	\end{split}
\end{align}
then we will have established the first property of \refD{SSNS} for $\mu$.

Toward this end,
\begin{align}\label{e:FFRPVR}
	\begin{split}
	\int_0^t &\int_\X (F(s, u), \Phi'(u)) \, d \mu_s(u) \, ds \\
		&= \int_0^t \int_\X (F(s, S(s) u), \Phi'(S(s) u)) \, d \mu_0(u) \, ds \\
		&= \int_0^t \int_\X \lim_{R \to \iny} (F_R(s, \TR S(s) u), \Phi'(S(s) u))  \, d \mu_0(u) \, ds \\
		&= \int_0^t \int_\X \lim_{R \to \iny}   (F_R(s, \TR S(s) u), \Phi_R'(\TR S(s) u))  \, d \mu_0(u) \, ds \\
		&= \lim_{R \to \iny}  \int_0^t \int_\X (F_R(s, \TR S(s) u), \Phi_R'(\TR S(s) u))  \, d \mu_0(u) \, ds.
	\end{split}
\end{align}
The first equality follows from \refL{mutTomuzero}.
The second equality follows from \refL{FBoundNoTime} and \refE{FBoundAll}, since $\Phi'(S(s) u)$ is bounded and compactly supported---and so also we can view the pairings as being in either the duality between $V$ and $V'$ or between $V_R$ and $V_R'$.
The third equality follows from \refE{PhipPhipR}.
The fourth equality follows from the dominated convergence theorem using \refE{FBoundAll}.

We would like to commute the roles of the projection operator and the solution operator in the right-hand side of \refE{FFRPVR} to allow us to apply \refL{IntXToHR}. To do this, we estimate,
\begin{align*}
	D(s, &u)
		= \vert
	(F_R(s, \TR S(s) u), \Phi_R'(\TR S(s) u)) \\
		&\qquad- (F_R(s, S_R(s) \TR u), \Phi_R'(S_R(s) \TR u))
	\vert \\
	&\le
	\vert
	(F_R(s, \TR S(s) u) - F_R(s, S_R(s) \TR u), \Phi_R'(\TR S(s) u))
	\vert \\
	  &\qquad+
	\vert
	(F_R(s, S_R(s) \TR u), \Phi_R'(\TR S(s) u) -  \Phi_R'(S_R(s) \TR u))
	\vert \\
	&\le
	\norm{F_R(s, \TR S(s) u) - F_R(s, S_R(s) \TR u)}_{V_R'} \norm{\Phi_R'(\TR S(s) u)}_{V_R} 
	  \\
	 &\qquad+
	\norm{F_R(s, S_R(s) \TR u)}_{V_R'} \norm{\Phi_R'(\TR S(s) u) -  \Phi_R'(S_R(s) \TR u)}_{V_R} .
\end{align*}

Letting
\begin{align}\label{e:hDef}
	h(R, u, s) := \norm{\Phi_R'(\TR S(s) u) -  \Phi_R'(S_R(s) \TR u)}_{V_R},
\end{align}
we have
\begin{align}\label{e:DIntBound}
	\begin{split}
	\int_0^t &\int_\X D(s, u) \, d \mu_0(u) \, ds \\
			&\le C \int_0^t \int_\X \norm{F_R(s, \TR S(s) u) - F_R(s, S_R(s) \TR u)}_{V'_R}  \, d \mu_0(u) \, ds \\
		&\qquad\qquad
			+ \int_0^t \int_\X \norm{F_R(s, S_R(s) \TR u)}_{V'_R} h(R, u, s) \, d \mu_0(u) \, ds.
	\end{split}
\end{align}
Applying the Cauchy-Schwarz inequality, we can bound the last term by
\begin{align*}
	 \int_\X &\norm{F_R(s, S_R(s) \TR u)}_{L^2([0, t]; V'_R)}
	 			\norm{h(R, u, \cdot)}_{L^2([0, t])}
	 	 \, d \mu_0(u) \\
		 	&\le C \int_\X 
	 			\norm{h(R, u, \cdot)}_{L^2([0, t])}
	 	 \, d \mu_0(u),
\end{align*}
using \refE{FBoundAll}.

By \refE{PhipRPhip},
\begin{align*}
	h(R, u, s) &= \norm{\Phi_R'(\TR S(s) u) -  \Phi_R'(S_R(s) \TR u)}_{V_R} \\
		&
		= \norm{\Phi'(\Cal{E}_R \TR S(s) u) -  \Phi'(\Cal{E}_R S_R(s) \TR u)}_V.
\end{align*}
By \refE{PhipBound}, $\norm{\Phi'(v)}_V \le C_0$ for all $v$ in $H_R$, for some $C_0$ independent of $R$. Thus, $h(R, u, \cdot) \le 2 C_0$. Also, because $\Phi' \colon H \to V$ is continuous it follows from \refE{ExpHR} that $h(R, u, \cdot) \to 0$ as $R \to \iny$ for all $u$ in $\Z$.
Hence, for all $u$ in $\Z$,
\begin{align*}
	\norm{h(R, u, \cdot)}_{L^2([0, t])}^2
		= \int_0^t h(R, u, s)^2 \, ds
		\to 0
\end{align*}
by the dominated convergence theorem. 
But then also $\norm{h(R, u, \cdot)}_{L^2([0, t])} \le 2 C_0 t^{1/2}$ and applying the dominated convergence theorem again gives
\begin{align*}
	\int_\X \norm{h(R, u, \cdot)}_{L^2([0, t])} \to 0
		\text{ as } R \to \iny.
\end{align*}
We conclude that the second term on the right-hand side of \refE{DIntBound} vanishes as $R \to \iny$.

For the first term in the right-hand side of \refE{DIntBound},
\begin{align*}
	&\norm{F_R(s, \TR S(s) u) - F_R(s, S_R(s) \TR u)}_{V'_R} \\
		&\qquad
		\le \norm{F_R(s, \TR S(s) u) - F(s, S(s)u)}_{V'_R} \\
		&\qquad\qquad\qquad
			+ \norm{F(s, S(s) u)- F_R(s, S_R(s) \TR u)}_{V'_R}.
\end{align*}
Since $d \mu_0(u) \, ds$ is a finite measure on $[0, t] \times \X$ and the first term on the right-hand side is both bounded and vanishes as $R \to \iny$ by \refL{FBoundNoTime} and \refE{FBoundAll}, after being integrated over $[0, t] \times \X$ the first term vanishes as $R \to \iny$.
The $L^2([0, t])$-norm of the second term on the right-hand side is bounded on the support of $\mu_0$ by \refE{FBoundAll} and vanishes as $R \to \iny$ by \refE{FDiffBound}; applying the Cauchy-Schwarz inequality followed by the dominated convergence theorem shows that
\begin{align*}
	&\int_0^t \int_\X \norm{F(s, S(s) u)- F_R(s, S_R(s) \TR u)}_{V'_R}   \, d \mu_0(u) \, ds \\
		&\qquad \le t^{1/2} \int_\X  \norm{F(s, S(s) u)- F_R(s, S_R(s) \TR u)}_{L^2([0, t]; V'_R)}
					\, d \mu_0(u)
\end{align*}
vanishes as $R \to \iny$.

We conclude that $D(s, u)$ integrates to zero in the limit as $R \to \iny$, meaning that
\begin{align*}
	\int_0^t &\int_\X (F(s, u), \Phi'(u)) \, d \mu_s(u) \, ds \\
		&=  \lim_{R \to \iny}\int_0^t \int_\X (F_R(s, S_R(s) \TR u), \Phi_R'(S_R(s) \TR u))  \, d \mu_0(u) \, ds.
\end{align*}

Finally, using \refL{IntXToHR} and \refL{mutTomuzero},
\begin{align*}
	 \int_0^t \int_\X &(F_R(s, S_R(s) \TR u), \Phi_R'(S_R(s) \TR u))  \, d \mu_0(u) \, ds \\
		&= \int_0^t \int_{H_R} (F_R(s, S_R(s) v), \Phi_R'(S_R(s) v))  \, d \mu_0^R(v) \, ds \\
		&= \int_0^t \int_{H_R} (F_R(s, v), \Phi_R'(v))  \, d \mu_s^R(v) \, ds,
\end{align*}
giving \refE{FToShow}, completing the demonstration that property (1) of \refD{SSNS} is satisfied for $\mu$.

The other properties in \refD{SSNS} follow more easily, using the dominated convergence theorem and the first two bounds in \refE{FBound}. Thus, we have established the existence of a SSNS for $R = \iny$ when the initial probability measure has bounded support in $\Z$. But we can drop this restriction by exploiting the inherent linearity in the definition of a SSNS, as done on p. 318 of \cite{FMRT}. This establishes the existence part of the theorem.

\textbf{Higher regularity}: We now add the assumption that the support of $\mu_0$ is $(\X, \Z)$-bounded. \refE{FMRT114Inf} and \refE{FMRT115Inf} follow much as did properties (2) through (5) of \refD{SSNS}. Adding the assumption that $f$ is time-independent, \refE{FMRT116Inf} follows for $R = \iny$ in the same way it does for $R < \iny$.

\textbf{Uniqueness}: The proof of uniqueness for $R < \iny$ on p. 319-321 of \cite{FMRT} applies with the following two changes: First, in the Galerkin approximation we use a basis for $\Z$ in place of the eigenfunctions of the Stokes operator (the spectrum no longer being discrete). Second, we use the energy bound in \refE{NSEnergyBoundGen} for $R = \iny$ in place of the bound involving the eigenvalue, $\la_m$, of the Stokes operator.
\end{proof}

\Ignore{ % Ignore
\begin{remark}\label{R:TRVersusPR4}

The proof of \refT{FMRTInf} would be simplified by using $\UR$ in place of $\TR$. For one, we would not need to make the initial assumption in \refE{mu0BoundedE1} (though this adds little complication). Second, $\UR$ has the property that $\UR u = u$ on $\Omega_{R/2}$. Together with the compact support of the test functions in $\Test$ gives immediate equality in certain integrals which we only showed held in the limit as $R \to \iny$. Third,
\end{remark}
} % End Ignore

We used the following two elementary lemmas in the proof of \refT{FMRTInf}. Note that when we say that equality holds between two integrals when the integrands are only Borel measurable, we mean that either both integrals are defined and equal or that both integrals are undefined. We state the lemmas this way because in their application we do not always know a priori that the integrands are integrable.

\begin{lemma}\label{L:IntXToHR}
	For any Borel measurable function $f$ on $H_R$,
	\begin{align}\label{e:IntxIntHR}
		\int_\X f(\TR u) \, d \mu_0(u)
			&= \int_{H_R} f(v) \, d \mu_0^R(v).
	\end{align}
\end{lemma}
\begin{proof}
First observe that $f \circ \TR$ is Borel measurable on $\X$ because $\TR$ is Borel measurable (in fact, continuous) and $f$ is Borel measurable, so the left-hand side of \refE{IntxIntHR} is well-defined. When $f = \chi_E$, the characteristic function of a Borel measurable subset $E$ of $H_R$,
\begin{align*}
	\int_\X f(\TR u) \, d \mu_0(u)
		&= \mu_0(\TRinv E)
		= \mu_0^R(E)
		= \int_{H_R} f(v) \, d \mu_0^R(v).
\end{align*}
\refE{IntxIntHR} then holds for simple functions by linearity, for nonnegative functions by the monotone convergence theorem, and hence for all Borel measurable functions.
\end{proof}

\begin{lemma}\label{L:mutTomuzero}
	For any function $f$ that is Borel measurable on $H_R$,
	\begin{align*}
		\int_{H_R} f(u) \, d \mu_t^R(u)
			&= \int_{H_R} f(S_R(t) u) \, d \mu_0^R(u).
	\end{align*}
	When $f$ is Borel measurable on $X$,
	\begin{align*}
		\int_X f(u) \, d \mu_t(u)
			&= \int_X f(S(t) u) \, d \mu_0(u).
	\end{align*}	
\end{lemma}
\begin{proof}
	As in the proof of \refL{IntXToHR}, equality holds for simple functions, then nonnegative functions,
	then all Borel measurable functions.
\end{proof}

\Ignore{ % Ignore
\begin{lemma}\label{L:IntConvergence}
Let $(g_n)$ be a sequence of measurable functions on the finite measure space $(A, \Cal{M}, \mu)$ that converges to 0 almost everywhere and that are uniformly bounded by $M$. Then
	\begin{align*}
		\int_A \abs{g_n} \, d \mu \to 0 \text{ as } n \to \iny.
	\end{align*}
\end{lemma}
\begin{proof}
Follows from the dominated convergence theorem, since $\mu$ is a finite measure.
%Fix $\eps > 0$ and let $E_k = \set{x \in A \colon \abs{g_n(x)} < \eps \text{ for all } n > k}$. Then
%\begin{align*}
%	\int_A \abs{g_n} \, d \mu
%		< \eps \mu(E_n)
%			+ M \mu(A \setminus E_n).
%\end{align*}
%But, $E_1 \subseteq E_2 \subseteq \cdots \subseteq A$ and $A = \cup_{k \ge 1} E_k$, so $\lim_{n \to \iny} \mu(E_n) = \mu(A)$ and $\lim_{n \to \iny} \mu(A \setminus E_n) = 0$. Therefore,
%\begin{align*}
%	\lim_{n \to \iny} \int_A \abs{g_n} \, d \mu
%		< \eps \lim_{n \to \iny} \mu(E_n)
%			+ M \lim_{n \to \iny} \mu(A \setminus E_n)
%			= \eps \lim_{n \to \iny} \mu(E_n).
%\end{align*}
%Since this is true for all $\eps > 0$, the conclusion of the lemma follows.
\end{proof}
} % End Ignore

%
% Section
%
\section{Construction of Euler solutions}\label{S:ConstructE}

\noindent We construct infinite-energy statistical solutions to the Euler equations by making a vanishing viscosity argument, using the infinite-energy statistical solutions to the Navier-Stokes equations that we constructed in \refS{ConstructNS}.

For initial velocities as in \refT{FMRTForEInf}, we have the following for SSNSs:
 \begin{theorem}\label{T:FMRTForNSInf}
 Assume that the support of the initial velocity $\mu_0$ for a SSNS with $R = \iny$ is bounded in $\Y$ as in \refD{BoundedSupport} and that $f$ is time-independent and lies in $\Y_0$. Then the SSNS also satisfies
\begin{align}\label{e:FMRT115E} 
	&\supp \mu_t \subseteq \set{u \in \Y \colon \norm{u}_{\Y} \le M(t)},
\end{align}
for a continuous function $M$ independent of $\nu$, and for all $p$ in $[p_0, \iny]$,
\Ignore{ % Ignore
\begin{align}\label{e:NSomegaIdentity}
	\begin{split}
	\int_\X &\norm{\omega}_{L^p}^p \, d \mu_t(u)
		+ p (p - 1) \int_0^t \int_\X \smallnorm{\abs{\omega}^{p/2 - 1} 
						\grad \omega}_{L^2(\R^2)}^2 \, d \mu_t(u) \\
		&= \int_\X \norm{\omega}_{L^p}^p \, d \mu_0(u) 
			+ p \int_0^t \int_\X (\omega(f), \abs{\omega}^{p - 2} \omega) \, d \mu_t(u) \, ds
	\end{split}
\end{align}
and
} % End Ignore
\begin{align}\label{e:EomegaBound}
	\int_\X \norm{\omega(u)}_{L^p} \, d \mu_t(u)
		\le \int_\X \norm{\omega(u)}_{L^p} \, d \mu_0(u)
			+ \int_0^t \norm{\omega(f(s))}_{L^p}  \, ds.
\end{align}
\end{theorem}
\begin{proof}
\Ignore{ % More detail on this calculation
For a deterministic solution $u(t) = S(t) u_0$ with $u_0$ in $\Y$ and $\omega = \omega(u)$,
\begin{align*} % \label{e:DetOmega}
	&\frac{d}{dt} \norm{\omega(t)}_{L^p}^p
		+ p (p - 1) \smallnorm{\abs{\omega}^{p/2 - 1} 
						\grad \omega}_{L^2(\R^2)}^2
		= p (\omega(f), \abs{\omega}^{p - 2} \omega),
\end{align*}
$p$ in $[p_0, \iny)$,
follows by taking the vorticity of \refE{NSp}, $\prt_t \omega + u \cdot \grad \omega = \nu \Delta \omega + \omega(f)$, multiplying both sides by $\abs{\omega}^{p - 2} \omega$, and integrating over space.
\Ignore{ % Ignore
Integrating \refE{DetOmega} over time gives
\begin{align}\label{e:DetOmegaIntt}
	\begin{split}
	&\norm{\omega(t)}_{L^p}^p
		+ p (p - 1) \int_0^t  \smallnorm{\abs{\omega}^{p/2 - 1} 
						\grad \omega}_{L^2(\R^2)}^2 \\
		&\qquad= \norm{\omega_0}_{L^p}^p 
			+ p \int_0^t (\omega(f), \abs{\omega}^{p - 2} \omega).
	\end{split}
\end{align}
} % End Ignore
This calculation is formal, but can be made precise using an approximation and smoothing argument (the result is classical).

Applying \Holders inequality gives
\begin{align*}
	\abs{(\omega(f), \abs{\omega}^{p - 2} \omega)}
		\le \norm{\omega(f)}_{L^p} \norm{\abs{\omega}^{p - 2} \omega}_{L^{p/(p - 1)}}
		= \norm{\omega(f)}_{L^p} \norm{\omega}_{L^p}^{p - 1},
\end{align*}
from which
\begin{align*}
	p \norm{\omega(t)}_{L^p}^{p-1} \diff{}{t} \norm{\omega(t)}_{L^p}
		\le p  \norm{\omega(f)}_{L^p} \norm{\omega}_{L^p}^{p - 1}
\end{align*}
follows.
Canceling the common factor and integrating over time gives
\begin{align*} % \label{e:DetOmegaBound}
	\norm{\omega(t)}_{L^p}
		\le \norm{\omega_0}_{L^p} + \int_0^t \norm{\omega(f)}_{L^p}.
\end{align*}

We can write this last inequality as
} % End Ignore

It is a standard result that
\begin{align}\label{e:EDetomegaBound}
	\norm{\omega(S(t) u)}_{L^p}
		\le \norm{\omega(u)}_{L^p} + \int_0^t \norm{\omega(f(t))}_{L^p}
\end{align}
for all $u$ in $\Y_0$. To prove it for $p = r/q$ in lowest terms, with $r$ even, one takes the vorticity of \refE{NSp}, $\prt_t \omega + u \cdot \grad \omega = \nu \Delta \omega + \omega(f)$, multiplies both sides by $\omega^{p - 1}$, and integrates over space and time formally to give
\begin{align*} % \label{e:DetOmegaIntt}
	&\norm{\omega(t)}_{L^p}^p
		+ p (p - 1) \int_0^t  \smallnorm{\omega^{p/2 - 1} 
						\grad \omega}_{L^2(\R^2)}^2 
		= \norm{\omega_0}_{L^p}^p 
			+ p \int_0^t (\omega(f), \omega^{p - 1}).
\end{align*}
An approximation and smoothing argument is required to establish the equality rigorously, and it then follows for all $p$ in $[p_0, \iny]$ by the continuity of the $L^p$ norm as a function of $p$. Applying \Holders inequality gives \refE{EDetomegaBound}.

Now assume that $u = \sigma_m + v$ is in $\Y_m$. Then $\prt_t u = \prt_t v$ and $\Delta u = \Delta v$ on $\Omega_1^C$, where $\Delta \sigma_m$ vanishes. Thus, the only additional complication in the argument above is the presence of the additional term $(\sigma_m \cdot \grad \omega, \omega^{p - 1}) = (1/p) (\sigma_m, \grad \omega^p)$. But this vanishes formally
by the divergence theorem, since $\sigma_m \cdot \mathbf{n} = 0$ on $\prt \Omega_R$, hence this term need not be accounted for in the approximation and smoothness argument.

Integrating \refE{EDetomegaBound} over $\X$ gives
\begin{align*}
	\int_\X \norm{\omega(S(t) u)}_{L^p} \, d \mu_0(u)
		\le \int_\X \norm{\omega(u)}_{L^p} \, d \mu_0(u)
		+ \int_0^t \norm{\omega(f(t))}_{L^p}.
\end{align*}
(The last term has no dependence on $u$ so the integral over $\X$ disappears, $\mu_0$ being a probability measure.) But $\norm{\omega(\cdot)}_{L^p} \colon \Y \to [0, \iny)$ is a bounded continuous function on $\supp \mu_0$ so \refE{EomegaBound} follows from $\mu_t = S(t) \mu_0$, and \refE{FMRT115E} follows from \refE{EomegaBound}.
\end{proof}

\Ignore{ % Ignore
In \refE{NSomegaIdentity}, the second term is defined because $\mu_t$ for $t > 0$ is supported on a subset of $\Y$ having sufficient regularity.
} % End Ignore

\Ignore{ % Ignore--probably won't need
From \refE{muBoundedXY}, $\mu_0$ is is $(\X, \Y)$-bounded means that
\begin{align*}
	\supp \mu_0 \subseteq \set{u \in \Y \colon \norm{u}_\X \le M},
		\text{ for some } M.
\end{align*}
} % End Ignore

\begin{theorem}\label{T:FMRTForEInf}
Assume that $\mu_0$ is supported in $\Y$ with
\begin{align*}
	 \int_\X \norm{u}_\X^2 \, d \mu_0(u) < \iny,
\end{align*}
and assume that $f$ is time-independent and lies in $\Y_0$. There exists a SSE, $\mu$, as in \refD{SSE}. One such solution is $\mu_t = \ol{S}(t) \mu_0$ for all $t \ge 0$, where $\ol{S}(t)$ is the solution operator for the two-dimensional Euler equations in $\R^2$ as in \refD{SolutionOperators}. Furthermore, if the support of $\mu_0$ is bounded in $\Y$ as in \refD{BoundedSupport} and $f$ is time-independent then this solution satisfies \refE{FMRT115E} for some function $M$ continuous on $[0, \iny)$ and \refE{EomegaBound}.
\Ignore{ % Ignore
and
\begin{align}\label{e:EomegaIdentity}
	\begin{split}
	\int_\X \norm{\omega}_{L^p}^p &\, d \mu_t(u)
		= \int_\X \norm{\omega}_{L^p}^p \, d \mu_0(u) \\
			&+ p \int_0^t \int_\X (\omega(f), \abs{\omega}^{p - 2} \omega) \, d \mu_t(u) \, ds.
	\end{split}
\end{align}
} % End Ignore
\end{theorem}
\begin{proof}
Assume first that the support of $\mu_0$ is bounded in $\Y$.
Define $\ol{\mu}_t = \ol{S}(t) \mu_0$, and let $\mu$ be the unique SSNS for $R = \iny$ given by \refT{FMRTInf} with the same forcing and initial data as for the Euler equations.
Let $\Phi = \phi((u, g_1), \dots, (u, g_k))$ lie in $\Test$. Then $g_1, \dots, g_k$ are in $V$ and
\begin{align*}
    \Phi'(u) &= \sum_{j=1}^k
        \prt_j \phi((u, g_1), \dots, (u, g_k)) g_j \in V, \\
    \grad \Phi'(u) &= \sum_{j=1}^k
        \prt_j \phi((u, g_1), \dots, (u, g_k)) \grad g_j \in L^2,
\end{align*}
with
\begin{align}\label{e:PhipBounds}
	\norm{\Phi'(u)}_V \le C, \quad
	\norm{\grad \Phi'(u)}_{L^2} \le C
\end{align}
for some constant $C$ independent of $u$ in $\X$.

Now,
\begin{align*}
	\int_\Y \Phi(u) \, d \mu_t(u)
		= \int_\Y \Phi(u) \, d \mu_0(u)
			+ \int_0^t \int_\Y (F(s, u), \Phi'(u)) \, d \mu_s(u) \, ds
\end{align*}
so, using $F = f - \nu A u - Bu$ and $\ol{F} = f - Bu$,
\begin{align*}
	\int_\Y &\Phi(u) \, d \mu_t(u)
		- \int_\Y \Phi(u) \, d \ol{\mu}_0(u)
		- \int_0^t \int_\Y (\ol{F}(s, u), \Phi'(u)) \, d \ol{\mu}_s(u) \, ds \\
        &= \int_\Y \Phi(u) \, d (\mu_0 - \ol{\mu}_0)(u) +
           \int_0^t \int_\Y (F(s, u) - \ol{F}(s, u), \Phi'(u))
                \, d \mu_s(u) \, ds \\
        &\qquad- \int_0^t \int_\Y (Bu, \Phi'(u))
                \, d (\mu_s - \ol{\mu}_s)(u) \, ds \\
       &= \int_\Y \Phi(u) \, d (\mu_0 - \ol{\mu}_0)(u) -
           \nu \int_0^t \int_\Y (Au, \Phi'(u))
                \, d \mu_s(u) \, ds \\
           &\qquad- \int_0^t \int_\Y (Bu, \Phi'(u))
                \, d (\mu_s - \ol{\mu}_s)(u) \, ds.
\end{align*}

But $\mu_0 = \ol{\mu}_0$, so
\begin{align*} %\label{e:StartingEquality}
    \begin{split}
        \int_\Y \Phi(u) \, &d (\mu_t - \ol{\mu}_t)(u)
            = - \nu \int_0^t \int_\Y (Au, \Phi'(u))
                    \, d \mu_s(u) \, ds \\
             &\qquad - \int_0^t \int_\Y (Bu, \Phi'(u))
                    \, d (\mu_s - \ol{\mu}_s)(u) \, ds.
    \end{split}
\end{align*}

We have,
\begin{align*}
    (Bu, \Phi'(u))
       &= (u \cdot \grad u, \Phi'(u))
\end{align*}
and
\begin{align*}
    (Au, \Phi'(u))
       &= -(\Delta u, \Phi'(u)) = (\grad u, \grad \Phi'(u)),
\end{align*}
since $\Phi'(u)$ is in $V$.
Thus,
\begin{align*}
    \begin{split}
        \int_\Y \Phi(u) \, &d (\mu_t - \ol{\mu}_t)(u)
            = - \nu \int_0^t \int_\Y (\grad u, \grad \Phi'(u))
                    \, d \mu_s(u) \, ds \\
                &\qquad
               - \int_0^t \int_\Y (u \cdot \grad u, \Phi'(u))
                    \, d (\mu_s - \ol{\mu}_s)(u) \, ds.
    \end{split}
\end{align*}

We have,
\begin{align*}
	 \int_\Y (\grad u, \grad \Phi'(u)) \, d \mu_s(u)
	 	\le  C \int_\Y \norm{\grad u}_{L^2} \, d \mu_s(u)
		\le C,
\end{align*}
where we used \refE{PhipBounds} followed by \refE{EomegaBound} and the boundedness of $\mu_0$ in $\Y$. The same bound holds when integrating against $\ol{\mu}_s$. Thus,
\begin{align}\label{e:RnutBound}
    \begin{split}
        &\abs{\int_\Y \Phi(u) \, d (\mu_t - \ol{\mu}_t)(u)} \\
           &\qquad \le R \nu t
               + \abs{ \int_0^t \int_\Y (u \cdot \grad u, \Phi'(u))
                    \, d (\mu_s - \ol{\mu}_s)(u) \, ds},
    \end{split}
\end{align}
where $R$ is proportional to the right-hand side of \refE{EomegaBound}, which we note increases with time unless there is zero forcing.

For any Borel measurable function $G$ on $H$,
\begin{align*}
    \int_\Y G(u) &d (\mu_s - \ol{\mu}_s)(u)
        =  \int_\Y G(u) d \mu_s(u)
            - \int_\Y G(u) d \ol{\mu}_s(u) \\
       &= \int_\Y G(S(s) u) \, d \mu_0(u)
            - \int_\Y G(\ol{S}(s) u) \, d \ol{\mu}_0(u) \\
        &= \int_\Y (G(S(s) u_0) - G(\ol{S}(s) u_0)) \, d \mu_0(u_0) \\
       &= \int_\Y (G(u(s)) - G(\ol{u}(s))) \, d \mu_0(u_0).
\end{align*}
In the last integral, we are defining $u(t)$ and $\ol{u}(t)$ to be
$S(t) u_0$ and $\ol{S}(t) u_0$, respectively. These are the solutions to
$(NS)$ and $(E)$ given the initial velocity $u_0$. (The support of
$\mu_0$ lying in $\Y$ insures that $\ol{S}(t) u_0$ is well-defined and continuous
for $\mu_0$-almost all $u_0$.)

Thus,
\begin{align*}
    \int_\Y &(u \cdot \grad u, \Phi'(u))
                    \, d (\mu_s - \ol{\mu}_s)(u) \\
       &= \int_\Y \brac{
            (u(s) \cdot \grad u(s), \Phi'(u(s))
            - (\ol{u}(s) \cdot \grad \ol{u}(s), \Phi'(\ol{u}(s)))} \, d \mu_0.
\end{align*}

Letting
$ % \begin{align*}
    w = u - \ol{u},
$ % \end{align*}
we have
\begin{align*}
	&(u \cdot \grad u, \Phi'(u)) -(\ol{u} \cdot \grad \ol{u}, \Phi'(\ol{u}))
		= (u \cdot \grad w, \Phi'(u)) \\
		&\qquad
			+ (u \cdot \grad \ol{u}, \Phi'(u) - \Phi'(\ol{u}))
			+ (w \cdot \grad \ol{u}, \Phi'(\ol{u})) \\
	  &\qquad = -(u \cdot \grad \Phi'(u), w)
			+ (u \cdot \grad \ol{u}, \Phi'(u) - \Phi'(\ol{u}))
			+ (w \cdot \grad \ol{u}, \Phi'(\ol{u})),
\end{align*}
so
\begin{align*}
	&\abs{
            (u(s) \cdot \grad u(s), \Phi'(u(s))
            - (\ol{u}(s) \cdot \grad \ol{u}(s), \Phi'(\ol{u}(s)))} \\
           &\qquad
           		\le \norm{u(s)}_{L^\iny} \norm{\grad \Phi'(u(s))}_{L^2} \norm{w(s)}_H \\
		&\qquad\qquad
					+ \norm{u(s)}_{L^\iny} \norm{\grad \ol{u}(s)}_{L^2} \norm{\Phi'(u(s)) - \Phi'(\ol{u}(s))}_H \\
		&\qquad\qquad
			+ \norm{w(s)}_H \norm{\grad \ol{u}(s)}_{L^2} \norm{\Phi'(\ol{u}(s))}_{L^\iny}.
\end{align*}

Now, \refE{EDetomegaBound} holds for solutions to ($E$): it can be derived as for ($NS$) or by viewing ($E$) as a non-homogeneous transport equation for the vorticity. Since $\supp \mu_0$ is bounded in $\Y$, it follows from \refE{EDetomegaBound} that $u$ and $\ol{u}$ are bounded in the $L^\iny([0, T] \times \R^2)$-norm uniformly over $\supp \mu_0$, as is $\grad \ol{u}$ in the $L^\iny([0, T]; L^2)$-norm. This is discussed more fully in \cite{C1996} or \cite{K2003}, where it is shown, moreover, that there exists a continuous function $\rho: [0, \iny) \times [0, \iny) \to [0, \iny)$, nondecreasing in $t$, with $\rho(0, t) = 0$ for all $t \ge 0$, such that for all $t > 0$,
\begin{align}\label{e:wrho}
    \norm{w(t)}_{H}
        \le \rho(\nu, t).
\end{align}
(For sufficiently small $\nu t$, $\rho(\nu, t) = (C \nu t)^{(1/2) e^{-Ct}}$.)

Also,
\begin{align*}
	&\norm{\Phi'(u(s)) - \Phi'(\ol{u}(s))}_H \\
		&\quad
			\le \sum_{j=1}^k \abs{\prt_j \phi((u(s), g_1), \dots, (u(s), g_k))
					- \prt_j \phi((u(s), g_1), \dots, (u(s), g_k))} \norm{g_j}_H.
\end{align*}
Now,
\begin{align*}
	\abs{(u(s), g_j) - (\ol{u}(s), g_j)}
		\le \norm{w(s)}_H \norm{g_j}_H
			\le \rho(\nu, s) \norm{g_j}_H,
\end{align*}
so since each $\prt_j \phi$ is continuous, it follows that
\begin{align*}
	\norm{\Phi'(u(s)) - \Phi'(\ol{u}(s))}_H
			\to 0 \text{ as } \nu \to 0
			\text{ uniformly over } [0, T].
\end{align*}
Combining all these facts shows that
\begin{align*}
	\int_0^t \int_\Y (u \cdot \grad u, \Phi'(u)) \, d (\mu_s - \ol{\mu}_s)(u)
			\to 0 \text{ as } \nu \to 0
\end{align*}
and hence that
\begin{align}\label{e:SSEEqL}
    \begin{split}
    \lim_{\nu \to 0} &\int_\Y \Phi(u) \, d \mu_t(u) \\
      &= \int_\Y \Phi(u) \, d \ol{\mu}_0(u)
          + \int_0^t \int_\Y (\ol{F}(s, u), \Phi'(u))
                \, d \ol{\mu}_s(u) \, ds.
    \end{split}
\end{align}
On the other hand,
\begin{align}\label{e:SSEEqR}
	\begin{split}
	\lim_{\nu \to 0} &\int_\Y \Phi(u) \, d \mu_t(u)
		= \lim_{\nu \to 0} \int_\Y \Phi(S(t) u) \, d \mu_0(u) \\
		&= \int_\Y \lim_{\nu \to 0}  \Phi(S(t) u) \, d \mu_0(u)
		= \int_\Y \Phi(\ol{S}(t) u) \, d \mu_0(u) \\
		&= \int_\Y \Phi(u) \, d \ol{\mu}_t(u).
	\end{split}
\end{align}
In the second equality we used the dominated convergence theorem. For the third equality, we used
\begin{align*}
	&\abs{\Phi(S(t) u) - \Phi(\ol{S}(t) u)} \\
		&\quad
		= \abs{\phi((S(t) u, g_1), \dots, (S(t) u, g_k))
			- \phi((\ol{S}(t) u, g_1), \dots, (\ol{S}(t) u, g_k))} \\
		&\quad
		\le \norm{\grad \phi}_{L^\iny} \abs{((S(t) u, g_1), \dots, (S(t) u, g_k))
			- ((\ol{S}(t) u, g_1), \dots, (\ol{S}(t) u, g_k))} \\
		&\quad
		\le C \abs{((S(t) u - \ol{S}(t) u, g_1), \dots, (S(t) u - \ol{S}(t) u, g_k))} \\
		&\quad
		\le C \norm{S(t) u - \ol{S}(t) u}_H
		\le C \rho(\nu, t) \to 0 \text{ as } \nu \to 0,
\end{align*}
the last inequality just being another way of writing \refE{wrho}. Hence, the right-hand sides of \refEAnd{SSEEqL}{SSEEqR} are equal, establishing the first property in \refD{SSE}.

\refEAnd{FMRT115E}{EomegaBound} follow as in the proof of \refT{FMRTForNSInf}.

As in the proof of \refT{FMRTInf}, we can drop the restriction that the support of $\mu_0$ is bounded in $\Y$ by exploiting the inherent linearity in the definition of a SSE, as done on p. 318 of \cite{FMRT}. The remaining properties in \refD{SSE} follow using the dominated convergence theorem in a manner similar to what we did above.  
\end{proof}

\Ignore{ % Ignore
We parallel the existence proof for an infinite-energy SSNS in \refS{ConstructNS}, but now using estimates obtained from the expanding domain limit in Theorem 8.1 of \cite{K2005ExpandingDomain} adapted to infinite-energy solutions. The role of the projection operator $\TR$ in \refS{ConstructNS} is now played by the operator $\YR$ of \refD{YR}.
Formally, \refEAnd{EomegaIdentity}{EomegaBound} follow by taking the vorticity of \refE{Ep}, $\prt_t \omega + u \cdot \grad \omega = \omega(f)$, multiplying both sides by $\abs{\omega}^{p - 2} \omega$, then integrating over space. Integrating the resulting identity over time gives \refE{EomegaIdentity}; applying \Holders inequality to the identity before integrating over time gives \refE{EomegaBound}. It is not hard to show that these formal derivations hold rigorously.
} % End Ignore

The proof of \refT{FMRTForEInf} shows that
\begin{align*}
	\int_\Y \Phi(u) \, d \mu_t(u)
		\to \int_\Y  \Phi(u) \, d\ol{\mu}_t(u)
			\text{ as } \nu \to 0.
\end{align*}
Since the space $\Test$ of test functions is dense in the space of all bounded continuous functions on $\X$, it follows that $\mu \to \ol{\mu}$ as measures as $\nu \to 0$; that is, the vanishing viscosity limit holds for statistical solutions to the Navier-Stokes and Euler equations.

\Ignore{ % No longer a separate section

%
% Section
%
\section{The vanishing viscosity limit of statistical solutions}\label{S:VV}

\noindent 

\begin{theorem}
Let $\mu_0$ be bounded in $\Y$ with $\mu$ the unique SSNS for $R = \iny$ with initial velocity $\mu_0$. Let $\ol{\mu}_t = \ol{S}(t) \mu_0$ be the SSE in the whole plane given by \refT{FMRTForEInf}. For simplicity, assume the same forcing for ($NS$) and ($E$). Then $\mu \to \ol{\mu}$ as measures as $\nu \to 0$.
\end{theorem}
\begin{proof}
Let $\Phi = \phi((u, g_1), \dots, (u, g_k))$ lie in $\Test$, so $g_1, \dots, g_k$ are in $V$ and
\begin{align*}
    \Phi'(u) &= \sum_{j=1}^k
        \prt_j \phi((u, g_1), \dots, (u, g_k)) g_j \in V \\
    \grad \Phi'(u) &= \sum_{j=1}^k
        \prt_j \phi((u, g_1), \dots, (u, g_k)) \grad g_j \in H,
\end{align*}
with
\begin{align}\label{e:PhipBounds}
	\norm{\Phi'(u)}_V \le C, \quad
	\norm{\grad \Phi'(u)}_H \le C
\end{align}
for some constant $C$ independent of $u$ in $\Y$.

Then,
\begin{align*}
    \int_\Y \Phi(u) \, d \mu_t(u)
       &= \int_\Y \Phi(u) \, d \mu_0(u)
          + \int_0^t \int_\Y (F(s, u), \Phi'(u))
                \, d \mu_s(u) \, ds, \\
    \int_\Y \Phi(u) \, d \ol{\mu}_t(u)
       &= \int_\Y \Phi(u) \, d \ol{\mu}_0(u)
          + \int_0^t \int_\Y (\ol{F}(s, u), \Phi'(u))
                \, d \ol{\mu}_s(u) \, ds,
\end{align*}
where $F = f - \nu A u - Bu$ and $\ol{F} = f - Bu$. Subtracting,
\begin{align*}
    \int_\Y &\Phi(u) \, d (\mu_t - \ol{\mu}_t)(u) \\
        &= \int_\Y \Phi(u) \, d (\mu_0 - \ol{\mu}_0)(u) +
           \int_0^t \int_\Y (F(s, u) - \ol{F}(s, u), \Phi'(u))
                \, d \mu_s(u) \, ds \\
        &\qquad- \int_0^t \int_\Y (Bu, \Phi'(u))
                \, d (\mu_s - \ol{\mu}_s)(u) \, ds \\
       &= \int_\Y \Phi(u) \, d (\mu_0 - \ol{\mu}_0)(u) -
           \nu \int_0^t \int_\Y (Au, \Phi'(u))
                \, d \mu_s(u) \, ds \\
           &\qquad- \int_0^t \int_\Y (Bu, \Phi'(u))
                \, d (\mu_s - \ol{\mu}_s)(u) \, ds.
\end{align*}

But $\mu_0 = \ol{\mu}_0$, so
\begin{align*} %\label{e:StartingEquality}
    \begin{split}
        \int_\Y \Phi(u) \, &d (\mu_t - \ol{\mu}_t)(u)
            = - \nu \int_0^t \int_\Y (Au, \Phi'(u))
                    \, d \mu_s(u) \, ds \\
             &\qquad - \int_0^t \int_\Y (Bu, \Phi'(u))
                    \, d (\mu_s - \ol{\mu}_s)(u) \, ds.
    \end{split}
\end{align*}

We have,
\begin{align*}
    (Bu, \Phi'(u))
       &= (u \cdot \grad u, \Phi'(u))
\end{align*}
and
\begin{align*}
    (Au, \Phi'(u))
       &= -(\Delta u, \Phi'(u)) = (\grad u, \grad \Phi'(u)),
\end{align*}
since $\Phi'(u)$ is in $V$.
Thus,
\begin{align*}
    \begin{split}
        \int_\Y \Phi(u) \, &d (\mu_t - \ol{\mu}_t)(u)
            = - \nu \int_0^t \int_\Y (\grad u, \grad \Phi'(u))
                    \, d \mu_s(u) \, ds \\
                &\qquad
               - \int_0^t \int_\Y (u \cdot \grad u, \Phi'(u))
                    \, d (\mu_s - \ol{\mu}_s)(u) \, ds.
    \end{split}
\end{align*}

We have,
\begin{align*}
	 \int_\Y (\grad u, \grad \Phi'(u)) \, d \mu_s(u)
	 	\le  C \int_\Y \norm{\grad u}_{L^2} \, d \mu_s(u)
		\le C,
\end{align*}
where we used \refE{PhipBounds} followed by \refE{EomegaBound} and the boundedness of $\mu_0$ in $\Y$. The same bound holds when integrating against $\ol{\mu}_s$. Thus,
\begin{align}\label{e:RnutBound}
    \begin{split}
        &\abs{\int_\Y \Phi(u) \, d (\mu_t - \ol{\mu}_t)(u)} \\
           &\qquad \le R \nu t
               + \abs{ \int_0^t \int_\Y (u \cdot \grad u, \Phi'(u))
                    \, d (\mu_s - \ol{\mu}_s)(u) \, ds},
    \end{split}
\end{align}
where $R$ is double the right-hand side of \refE{EomegaBound}, which we note increases with time unless there is zero forcing.

To make any further progress, we need to introduce the solution
operators, which makes things specific to two dimensions. So let
$(S(t))_{t \ge 0}$ and $(\ol{S}(t))_{t \ge 0}$ be the solution
operators for $(NS)$ and $(E)$, respectively. Then for any measurable
function $G$ on $H$,
\begin{align*}
    \int_\Y G(u) &d (\mu_s - \ol{\mu}_s)(u)
        =  \int_\Y G(u) d \mu_s(u)
            - \int_\Y G(u) d \ol{\mu}_s(u) \\
       &= \int_\Y G(S(s) u) \, d \mu_0(u)
            - \int_\Y G(\ol{S}(s) u) \, d \ol{\mu}_0(u) \\
        &= \int_\Y (G(S(s) u_0) - G(\ol{S}(s) u_0)) \, d \mu_0(u_0) \\
       &= \int_\Y (G(u(s)) - G(\ol{u}(s))) \, d \mu_0(u_0).
\end{align*}
In the last integral, we are defining $u(t)$ and $\ol{u}(t)$ to be
$S(t) u_0$ and $\ol{S}(t) u_0$, respectively. These are the solutions to
$(NS)$ and $(E)$ given the initial velocity $u_0$. (The support of
$\mu_0$ lying in $\Y$ insures that $\ol{S}(t) u_0$ is well-defined and continuous
for $\mu_0$-almost all $u_0$.)

Thus,
\begin{align*}
    \int_\Y &(u \cdot \grad u, \Phi'(u))
                    \, d (\mu_s - \ol{\mu}_s)(u) \\
       &= \int_\Y \brac{
            (u(s) \cdot \grad u(s), \Phi'(u(s))
            - (\ol{u}(s) \cdot \grad \ol{u}(s), \Phi'(\ol{u}(s)))} \, d \mu_0.
\end{align*}

Letting
\begin{align*}
    w(t) = u - \ol{u},
\end{align*}
we have
\begin{align*}
	&(u \cdot \grad u, \Phi'(u)) -(\ol{u} \cdot \grad \ol{u}, \Phi'(\ol{u}))
		= (u \cdot \grad w, \Phi'(u)) \\
		&\qquad
			+ (u \cdot \grad \ol{u}, \Phi'(u) - \Phi'(\ol{u}))
			+ (w \cdot \grad \ol{u}, \Phi'(\ol{u})) \\
	  &\qquad = -(u \cdot \grad \Phi'(u), w)
			+ (u \cdot \grad \ol{u}, \Phi'(u) - \Phi'(\ol{u}))
			+ (w \cdot \grad \ol{u}, \Phi'(\ol{u})),
\end{align*}
so
\begin{align*}
	&\abs{
            (u(s) \cdot \grad u(s), \Phi'(u(s))
            - (\ol{u}(s) \cdot \grad \ol{u}(s), \Phi'(\ol{u}(s)))} \\
           &\qquad
           		\le \norm{u(s)}_{L^\iny} \norm{\grad \Phi'(u(s))}_H \norm{w(s)}_H \\
		&\qquad\qquad
					+ \norm{u(s)}_{L^\iny} \norm{\grad \ol{u}(s)}_{L^2} \norm{\Phi'(u(s)) - \Phi'(\ol{u}(s))}_H \\
		&\qquad\qquad
			+ \norm{w(s)}_H \norm{\grad \ol{u}(s)}_{L^2} \norm{\Phi'(\ol{u}(s))}_{L^\iny}.
\end{align*}

Since $\supp \mu_0$ is bounded in $\Y$, it follows from \cite{C1996} or \cite{K2003} that $u$ and $\ol{u}$ are bounded in the $L^\iny([0, T] \times \R^2)$-norm uniformly over $\supp \mu_0$, as is $\grad \ol{u}$ in the $L^\iny([0, T]; L^2)$-norm (these facts also follow from \refE{EomegaBound}) and that for all $t > 0$,
\begin{align}\label{e:wrho}
    \norm{w(t)}_{H}
        \le \rho(\nu, t),
\end{align}
for some continuous function $\rho: [0, \iny) \times [0, \iny) \to [0, \iny)$,
nondecreasing in $t$, with $\rho(0, t) = 0$ for all $t \ge 0$.

Also,
\begin{align*}
	&\norm{\Phi'(u(s)) - \Phi'(\ol{u}(s))}_H \\
		&\quad
			\le \sum_{j=1}^k \abs{\prt_j \phi((u(s), g_1), \dots, (u(s), g_k))
					- \prt_j \phi((u(s), g_1), \dots, (u(s), g_k))} \norm{g_j}_H.
\end{align*}
Now,
\begin{align*}
	\abs{(u(s), g_j) - (\ol{u}(s), g_j)}
		\le \norm{w(s)}_H \norm{g_j}_H
			\le \rho(\nu, s) \norm{g_j}_H,
\end{align*}
so since each $\prt_j \phi$ is continuous, it follows that
\begin{align*}
	\norm{\Phi'(u(s)) - \Phi'(\ol{u}(s))}_H
			\to 0 \text{ as } \nu \to 0
			\text{ uniformly over } [0, T].
\end{align*}
Combining all these facts shows that
\begin{align*}
	\int_0^t \int_\Y (u \cdot \grad u, \Phi'(u)) \, d (\mu_s - \ol{\mu}_s)(u)
			\to 0 \text{ as } \nu \to 0
\end{align*}
and hence by \refE{RnutBound} that
\begin{align*}
	\int_\Y \Phi(u) \, d \mu_t(u)
		\to \int_\Y  \Phi(u) \, d\ol{\mu}_t(u)
			\to 0 \text{ as } \nu \to 0.
\end{align*}
Since the space $\Test$ of test functions is dense in the space of all bounded continuous functions on $\X$, it follows that $\mu \to \ol{\mu}$ as measures as $\nu \to 0$.
\end{proof}

From the proof, we could write down explicitly the rate of convergence of $\int_\Y \Phi(u) \, d \mu_t(u)$ to $ \int_\Y \Phi(u) \, d \ol{\mu}_t(u)$, the rate depending on the essential supremum of $\norm{u}_\Y$ over $\supp \mu_0$, $\norm{\Phi}_{C^1}$, and the modulus of continuity of $\Phi'$.

The vanishing viscosity argument above can be adapted to prove the existence result in \refT{FMRTForEInf} as follows:

\begin{proof}[\textbf{Alternate proof of \refT{FMRTForEInf}}]
Assume first that the support of $\mu_0$ is bounded in $\Y$ and $f$ is time-independent.
From what we showed above, if we let $\mu_t = \ol{S}(t) \mu_0$ then
\begin{align*}
    \lim_{\nu \to 0} &\int_\Y \Phi(u) \, d \mu_t(u) \\
      &= \int_\Y \Phi(u) \, d \ol{\mu}_0(u)
          + \int_0^t \int_\Y (\ol{F}(s, u), \Phi'(u))
                \, d \ol{\mu}_s(u) \, ds.
\end{align*}
On the other hand,
\begin{align*}
	\lim_{\nu \to 0} &\int_\Y \Phi(u) \, d \mu_t(u)
		= \lim_{\nu \to 0} \int_\Y \Phi(S(t) u) \, d \mu_0(u) \\
		&= \int_\Y \lim_{\nu \to 0}  \Phi(S(t) u) \, d \mu_0(u)
		= \int_\Y \Phi(\ol{S}(t) u) \, d \mu_0(u) \\
		&= \int_\Y \Phi(u) \, d \ol{\mu}_t(u).
\end{align*}
In the second equality we used the dominated convergence theorem. For the fourth inequality, we used
\begin{align*}
	&\abs{\Phi(S(t) u) - \Phi(\ol{S}(t) u} \\
		&\quad
		= \abs{\phi((S(t) u, g_1), \dots, (S(t) u, g_k))
			- \phi((\ol{S}(t) u, g_1), \dots, (\ol{S}(t) u, g_k))} \\
		&\quad
		\le \norm{\phi'}_{L^\iny} \abs{((S(t) u, g_1), \dots, (S(t) u, g_k))
			- ((\ol{S}(t) u, g_1), \dots, (\ol{S}(t) u, g_k))} \\
		&\quad
		\le C \abs{((S(t) u - \ol{S}(t) u, g_1), \dots, (S(t) u - \ol{S}(t) u, g_k))} \\
		&\quad
		\le C \norm{S(t) u - \ol{S}(t) u}_H
		\le C \rho(\nu, t) \to 0 \text{ as } \nu \to 0,
\end{align*}
the last inequality just being another way of writing \refE{wrho}.

This gives the first property in \refD{SSE}. As in the proof of \refT{FMRTInf}, we can drop the restriction that the support of $\mu_0$ is bounded in $\Y$ by exploiting the inherent linearity in the definition of a SSNS, as done on p. 318 of \cite{FMRT}. The remaining properties in \refD{SSE} follow using the dominated convergence theorem in a manner similar to what we did above, as does \refE{EomegaIdentity} (and hence \refE{EomegaBound}).  
\end{proof}
} % End Ignore

%
% Section---conditionally included
%
\ifthenelse{\value{bIncludeConstantinRamosSection}=0}
{}
{
\section{Damped and driven fluid equations}\label{S:DD}

\noindent In \cite{CR2007}, Constantin and Ramos consider stationary statistical solutions to the damped and driven Navier-Stokes equations (SSSNS) and their inviscid limit to a stationary statistical solution to the damped and driven Euler equations (SSSE). We refer the reader to \cite{CR2007} for definitions of such solutions. We restrict ourselves to the following brief observations.

The deterministic form of the time-dependent damped and driven Navier-Stokes equations (Equation (1) of \cite{CR2007}) is
\begin{align}\label{e:DDNSu}
	\prt_t u + u \cdot \grad u - \nu \Delta u + \gamma u + \grad p = f,
		\quad \dv u = 0.
\end{align}
Taking the vorticity of this equation gives Equation (4) of \cite{CR2007}:
\begin{align}\label{e:DDNSomega}
	\prt_t \omega + u \cdot \grad \omega - \nu \Delta \omega + \gamma \omega = g.
\end{align}
Here, $g$ is the vorticity of the vector-valued forcing term, $f$.

In constructing a SSSNS from long time-averages of a solution to \refE{DDNSu}, Constantin and Ramos assume (in Section 5 of \cite{CR2007}) that $f$ belongs to $W^{1, 1} \cap W^{1, \iny}$, which means that $g$ belongs to $L^1 \cap L^\iny$ with
\begin{align*}
	\int_{\R^2} g = 0.
\end{align*}
This last assumption is not stated explicitly, but follows from their assumption that $g = \omega(f)$ is in $L^1$ with $f$ in $L^2$ (see, for instance, \cite{C1998} p. 12). Similarly, $u_0$ is assumed to lie in $L^2$ with $\omega_0 = \omega(u_0)$ in $L^1 \cap L^\iny$, from which it follows that
\begin{align*}
	\int_{\R^2} \omega_0 = 0.
\end{align*}

At this point we do not mean \refE{DDNSomega} to be the vorticity formulation of \refE{DDNSu}, since we can, a priori, only recover a velocity field in $L^\iny$ via the Biot-Savart law (see \refC{CorBSLaw}). The vorticity formulation is not made in \cite{CR2007} until the SSSNS is introduced, and in that setting we do not need the additional regularity of $u$.

Integrating \refE{DDNSomega} over all of $\R^2$ and letting $m(t) = \int_{\R^2} \omega(t)$ and $\eta = \int_{\R^2} g$, it follows formally that
\begin{align*}
	m'(t) + \gamma m(t)
			= \eta
\end{align*}
(this formal calculation can be made rigorous). Thus,
\begin{align}\label{e:m}
	m(t)
		= \frac{\eta}{\gamma} + \pr{m(0) - \frac{\eta}{\gamma}} e^{-\gamma t}.
\end{align}
But $m(0) = \eta = 0$, so $m(t)$ is zero for all time. (Of course, this conclusion follows more immediately from $u(t)$ lying in $L^2$ with $\omega(t)$ in $L^1 \cap L^\iny$, but we will extend it to nonzero values of $m(t)$ in a moment.)

Now, if we wish to extend the construction in Section 5 of \cite{CR2007} to cover initial velocities with infinite energy, and so be able to consider vortex patch initial data, \refE{m} will still hold if we assume that $u_0$ lies in $\Y$ of \refE{Y} (with $p_0 = 1$). We may, however, have $m(0) \ne 0$ while $\eta = 0$, so $m(t) = m(0) e^{- \gamma t}$. It follows that for any long time-average, using the notation of \cite{CR2007},
\begin{align*}
	\Lim_{t \to \iny} &\frac{1}{t} \int_0^t
		m(S^{NS, \gamma} (s + t_0) \omega_0) \, ds \\
		&= \lim_{t \to \iny} \frac{1}{t} \int_0^t
			m(S^{NS, \gamma} (s + t_0) \omega_0) \, ds
		= 0.
\end{align*}
Thus, the stationary solution that we recover is still finite-energy.

The change that makes the constructed SSSNS infinite-energy is simple: require that $\eta = \gamma m(0)$. For then by \refE{m}, $m(t) = m(0)$ for all time and hence also in the long time-average. This fact merely reflects the observation that integrating the vorticity for a stationary deterministic solution (Equation (6) of \cite{CR2007}) over all of $\R^2$ requires that $\gamma \int_{\R^2} \omega = \int_{\R^2} g$. A limitation, though, of such a construction is that the solution is supported on a subset of $\Y_m$ for one fixed $m$, and hence exhibits no variation in how ``infinite'' its energy is. A possible remedy to this is to allow time-varying forcing, making $\eta$ vary over time in such a way as to make the stationary solution spread across different $\Y_m$ spaces.

\begin{remark}\label{R:InfiniteFiniteEnergyCR}
Of course, it is a little misleading to even say that the SSSNSs constructed in Theorem 5.1 of \cite{CR2007} are finite-energy. More precisely, what is true is that the total mass of the vorticity is zero $\mu^\nu$-a.e., because
\begin{align*}
	\int_{L^2(\R^2)}
		\pr{\int_{\R^2} \omega(x) \, dx}^2 \, d \mu^\nu(\omega)
			= 0.
\end{align*}
This follows from a limiting argument using the test functions $\Psi^k(\omega) = \psi((\omega, w^k))$, where $\psi$ in $C^\iny(\R)$ is chosen so that $\psi(x) = x^2$ for $\abs{x} \le 1$ and each $w_k$ is a smoothed compactly supported version of the characteristic function of the ball with radius $k$ centered at the origin.
\end{remark}

\Ignore{ % Ignore
\begin{cor}\label{L:omegaMap}
The vorticity operator $\omega$ maps $\Y$ (with $p_0 = 1$) injectively into $L^1 \cap L^\iny(\R^2)$. On $\omega(\Y)$, $K* = \omega^{-1}$; that is, \refE{BSLaw} applied to a vorticity in $\omega(\Y)$ returns the unique vector in $\Y$ having the given vorticity.
\end{cor}
\begin{proof}
That $\omega$ maps $\Y$ into $L^1 \cap L^\iny(\R^2)$ is clear. To prove injectivity, suppose that $u_1 = \sigma_m + v_1$ lies in $E_m \cap \dot{W}^{1, 1} \cap \dot{W}^{1, \iny}(\R^2)$ and $u_2 = \sigma_{m'} + v_2$ lies in $E_{m'} \cap \dot{W}^{1, 1} \cap \dot{W}^{1, \iny}(\R^2)$ with $\omega(u_1) = \omega(u_2) = \omega$. Then $u_1 - u_2 = \sigma_{m - m'} + (v_1 - v_2)$ satisfies  $\omega(u_1 - u_2) = 0$. It follows form Lemma 1.3.1 p. 12 of \cite{C1998} that $u_1 - u_2 = 0$. That $K* = \omega^{-1}$ then follows from \refL{BSLaw}.
\end{proof}
} % End Ignore

} % End Constantin Ramos section

%
% Section
%
\section*{Acknowledgements}

The author was supported in part by NSF grant DMS-0705586 during the period of this work.

% \Ignore{ % Ignore

% } % End Ignore

% \bibliography{Refs}
% \bibliographystyle{plain}

\end{document}